\documentclass[10pt]{amsart}

\pdfoutput=1
\usepackage[text={420pt,660pt},centering]{geometry}

\usepackage{cancel}
\usepackage{graphicx}
\IfFileExists{MnSymbol.sty}{\usepackage{MnSymbol}}{\usepackage{amssymb}}
\usepackage{mathtools} 
\usepackage[colorlinks=true, pdfstartview=Fit, linkcolor=blue, citecolor=blue, urlcolor=blue,pagebackref=false]{hyperref}
\usepackage{orcidlink}
\usepackage{microtype}

\usepackage{bm}
\IfFileExists{dsfont.sty}{\usepackage{dsfont}}{}

\usepackage{mathrsfs}
\usepackage{xcolor}
\IfFileExists{accents.sty}{\usepackage{accents}}{}
\providecommand{\accentset}[2]{\overset{#1}{#2}}

\IfFileExists{ulem.sty}{\usepackage[normalem]{ulem}}{}

\definecolor{darkgreen}{rgb}{0,0.5,0}
\definecolor{darkblue}{rgb}{0,0,0.7}
\definecolor{darkred}{rgb}{0.9,0.1,0.1}

\newtheorem{proposition}{Proposition}
\newtheorem{theorem}[proposition]{Theorem}
\newtheorem{lemma}[proposition]{Lemma}
\newtheorem{corollary}[proposition]{Corollary}

\newtheorem{conjecture}[proposition]{Conjecture}
\theoremstyle{remark}

\theoremstyle{definition}
\newtheorem{definition}[proposition]{Definition}

\numberwithin{equation}{section}
\numberwithin{proposition}{section}
\numberwithin{figure}{section}
\numberwithin{table}{section}

\newcommand{\N}{\mathbb{N}}

\newcommand{\R}{\mathbb{R}}

\newcommand{\E}{\mathbb{E}}
\renewcommand{\P}{\mathbb{P}}

\newcommand{\B}{\mathsf{B}}

\renewcommand{\le}{\leqslant}
\renewcommand{\ge}{\geqslant}
\renewcommand{\leq}{\leqslant}
\renewcommand{\geq}{\geqslant}

\renewcommand{\subset}{\subseteq}
\renewcommand{\bar}{\overline}

\newcommand{\Ll}{\left}
\newcommand{\Rr}{\right}

\newcommand{\1}{\mathbf{1}}

\newcommand{\mcl}{\mathcal}

\newcommand{\mfk}{\mathfrak}

\DeclareMathOperator{\diam}{diam}
\DeclareMathOperator{\tr}{tr}
\DeclareMathOperator{\supp}{supp}

\newcommand{\cQ}{\mathcal{Q}}

\renewcommand*{\dot}[1]{%
\accentset{\mbox{\large\bfseries .}}{#1}}

\newcommand{\bx}{\mathbf{x}}
\newcommand{\bg}{\mathbf{g}}

\newcommand{\bz}{\mathbf{z}}
\newcommand{\bX}{\mathbf{X}}
\newcommand{\bY}{\mathbf{Y}}
\newcommand{\dd}{\,\mathrm d}
\DeclareMathOperator{\Ent}{Ent}
\newcommand{\Cov}{\mathsf C}
\newcommand{\overlap}{\mathsf R}
\DeclareMathOperator{\Var}{Var}

\newcommand{\norm}[1]{\lVert#1\rVert}
\newcommand{\ip}[2]{\langle#1,#2\rangle}
\newcommand{\stackref}[2]{\stackrel{#2}{#1}}

\newcommand{\cR}{\mathfrak R}
\newcommand{\bb}{\mathcal B_b}
\newcommand{\cP}{\mathcal P}

\author[H.-B.\ Chen]{Hong-Bin Chen\,\orcidlink{0000-0001-6412-0800}}
\address[Hong-Bin Chen]{NYU-ECNU Institute of Mathematical Sciences, NYU Shanghai, China}
\email{\href{mailto:hongbin.chen@nyu.edu}{hongbin.chen@nyu.edu}}

\author[F.-H. \ Ho]{Fu-Hsuan Ho\,\orcidlink{0000-0003-3178-5472}}
\address[Fu-Hsuan Ho]{Weizmann Institute of Science, Rehovot, Israel}
\email{\href{mailto:fu-hsuan.ho@weizmann.ac.il}{fu-hsuan.ho@weizmann.ac.il}}

\author[V.\ Issa]{Victor Issa\,\orcidlink{0009-0009-1304-046X}}
\address[Victor Issa]{LPSM, Sorbonne University, Paris, France}
\email{\href{mailto:victor.issa@ens-lyon.fr}{victor.issa@ens-lyon.fr}}

\begin{document}
\title[Free energy of non-convex vector spin glasses]{A formula for the free energy of non-convex vector spin glasses}
\date{October 4, 2026}
\begin{abstract}

We consider vector mean-field spin glass models with independent and identically distributed spins without assuming the covariance function is convex.
We identify the limit free energy as the minimum of a functional that can be written as the sum of a relative entropy and the solution of a Hamilton--Jacobi equation. 
The proof of this result relies on the observation that when the spin configurations of the system are constrained to have a given empirical measure, the corresponding free energy approximately satisfies a Hamilton--Jacobi equation 
with a uniformly bounded error. 
\end{abstract}
\maketitle
\section{Introduction}

\subsection*{The model}
Write $\R_+:=[0,\infty)$. Fix $D\in\N$, put $\B:=\{a\in\R^D:|a|\le1\}$, and let $\mu$ be a probability measure supported on $\B$. For $N\in\N$, we view a spin configuration $\sigma\in(\supp\mu)^N$ as a $D$-by-$N$ matrix, with columns $\sigma_{\cdot,i}\in\supp\mu$ for $i\in\{1,\dots,N\}$. For matrices or vectors $a$ and $b$ of the same dimension, let $a\cdot b:=\tr(ab^\intercal)$ and $|a|:=(a\cdot a)^{1/2}$. The overlap between two configurations $\sigma$ and $\tau$ is defined by
\begin{equation*}
 \overlap_N(\sigma,\tau):=\frac{\sigma\tau^\intercal}{N}.
\end{equation*}
Let $\xi:\R^{D\times D}\to\R$ be smooth, admit a power-series expansion converging absolutely at every matrix in $\R^{D\times D}$, and satisfy $\xi(0)=0$. For every $N\in\N$, 
we consider the Hamiltonian $H_N:\R^{D\times N}\to\R$ which is a centered Gaussian process with covariance
\begin{equation}\label{eq:covariance}
 \E H_N(\sigma)H_N(\tau)=N\xi\bigl(\overlap_N(\sigma,\tau)\bigr),\qquad \sigma,\tau\in\R^{D\times N}.
\end{equation}
These are the covariance assumptions in~\cite[Sections 1.1 and 1.5]{chen2025free}. We do not assume convexity of $\xi$.

For $N\in\N$ and $t\in\R_+$, define the free energy
\begin{equation}\label{eq:free-energy-zero}
 \bar F_{N,\mu^{\otimes N}}(t,0):=-\frac1N\E\log\int\exp\left(\sqrt{2t}\,H_N(\sigma)-tN\xi\bigl(\overlap_N(\sigma,\sigma)\bigr)\right)\mu^{\otimes N}(\mathrm d\sigma).
\end{equation}
The self-overlap correction term $-tN\xi\bigl(\overlap_N(\sigma,\sigma)\bigr)$ is introduced as a convenience to simplify the expression of the free energy when differentiating with respect to $t$. Theorem~\ref{thm:main} identifies the limiting free energy; the limit without this correction can be recovered by the procedure in~\cite{chen2024selfoverlap}.

\subsection*{The main result}

Let $S^D$ be the space of real symmetric $D$-by-$D$ matrices and $S^D_+ \subset S^D$ the cone of positive semi-definite matrices. Given $a,b \in S^D$, we write $a \preceq b$ when $b-a \in S^D_+$. We denote by $\cQ$ the set of right-continuous increasing paths $q:[0,1)\to S^D_+$, where increasing means that $q(r)\preceq q(r')$ for every $0\le r\le r'<1$. For $r\in[1,\infty]$, let
\begin{equation*}
 \cQ_r:=\cQ\cap L^r([0,1];S^D).
\end{equation*}
In Section~\ref{sec:enriched}, we define an enriched free energy $\bar F_{N,\lambda_N}(t,q)$ for every nonzero finite measure $\lambda_N$ on $\B^N$ and every $q\in\cQ_1$. In particular, $\bar F_{N,\mu^{\otimes N}}(t,q)$ denotes the enriched free energy associated with the original product reference measure $\mu^{\otimes N}$. At $q=0$, this definition agrees with~\eqref{eq:free-energy-zero}.

Write $\cP(\supp\mu)$ for the probability measures on $\supp\mu$ and $\bb(\supp\mu)$ for the bounded Borel functions from $\supp\mu$ to $\R$. For $\nu\in\cP(\supp\mu)$ and $h\in\bb(\supp\mu)$, let $e^h\nu$ denote the finite measure with density $e^h$ relative to $\nu$. Define $\psi_{\nu,h}:\cQ_1\to\R$ and $\varphi^\nu:\cQ_1\to\R$ by
\begin{align}\label{eq:initial-definition}
 \psi_{\nu,h}(q):=\bar F_{1,e^h\nu}(0,q),\qquad\text{and}\qquad
 \varphi^\nu(q):=\sup_{h\in\bb(\supp\mu)}\left\{\psi_{\nu,h}(q)+\int_{\supp\mu}h(a)\,\nu(\mathrm da)\right\}.
\end{align}

For $\nu\in\cP(\supp\mu)$, let $f^\nu:\R_+\times\cQ_2\to\R$ denote the solution of the following Hamilton--Jacobi equation
\begin{equation}
\label{eq:HJ-inf-dim}
    \begin{cases}
         \partial_t f^\nu(t,q)-\int_0^1\xi\bigl(\partial_q f^\nu(t,q)(s)\bigr)\dd s=0 \\
         f^\nu(0,q)=\varphi^\nu(q).
    \end{cases}
\end{equation}
The precise meaning of this equation is explained in Section~\ref{sec:constrained-HJ}. In particular, the notion of solution that we use is the notion of viscosity solution; see Definition~\ref{def:viscosity} and Definition~\ref{def:spin-glass-solution}.

Denote by $\Ent(\cdot\mid\mu):\cP(\supp\mu)\to[0,\infty]$ the relative entropy
\begin{equation}
    \label{eq:relative-entropy}
 \Ent(\nu\mid\mu):=\begin{cases}
 \displaystyle\int_{\supp\mu}\log\!\left(\frac{\mathrm d\nu}{\mathrm d\mu}(a)\right)\nu(\mathrm da),&\nu\ll\mu,\\
 +\infty,&\nu\not\ll\mu.
 \end{cases}
\end{equation}

\begin{theorem}\label{thm:main}
For every $t\in\R_+$ and $q\in\cQ_2$, we have
\begin{equation}\label{eq:main}
 \lim_{N\to\infty}\bar F_{N,\mu^{\otimes N}}(t,q)=\min_{\nu\in\cP(\supp\mu)}\left\{\Ent(\nu\mid\mu)+f^\nu(t,q)\right\}.
\end{equation}
The minimum is attained. For every $T,Q\in\R_+$, the convergence in \eqref{eq:main} is uniform over $t\in[0,T]$ and $q\in\cQ_\infty$ with $|q|_{L^\infty}\le Q$.
\end{theorem}


\subsection*{Background and related works}

The Sherrington--Kirkpatrick model corresponds to choosing $D = 1$, taking $\mu$ to be the uniform measure on $\pm 1$ and setting $\xi(x) = x^2$. The Parisi formula was derived in the physics literature in the context of the Sherrington--Kirkpatrick model using the non-rigorous replica method~\cite{parisi79,parisi80}. This formula expresses the free energy of the model as the maximum value of a functional defined on $\mcl P(\R_+)$. A rigorous proof of the lower bound in the Parisi formula is given in~\cite{gue03} using a technique now known as Guerra's interpolation. A rigorous proof of the upper bound was then obtained in~\cite{Tpaper}, completing the proof of the Parisi formula. 

In~\cite{pan.aom} it is proven that the limiting Gibbs measure admits some structural ultrametric property. This made it possible to simplify the previous proof of the Parisi formula and to generalize it to the models we study in the present paper with the additional assumption that $\xi$ is convex~\cite{pan,pan.vec}. The $D>1$ case also brought to light the synchronization mechanism, first discovered in multi-species mean-field spin glasses \cite{pan.multi} as introduced in~\cite{barcon}.


Related results have also been obtained in the context of models with spherical spins both in the one-species case~\cite{tal.sph} and the multi-species case~\cite{bates2022free}.

Finding expressions for the limit free energy beyond the cases where $\xi$ is convex is a difficult task. Some progress has been made in the context of balanced multi-species spin glasses~\cite{bates2023parisi,bates2025balanced} as in some cases the free energy of the full system reduces to the free energy of a system with $D=1$ and $\xi$ convex. However, no Parisi type formula is known in general for models with nonconvex $\xi$ and some of those formulas are even provably false.

The Hamilton--Jacobi approach to spin glasses was introduced as an alternative method to compute and characterize the limit free energy. In essence, this approach consists in connecting the free energy with the solution of a Hamilton--Jacobi PDE. It was first developed in the high-temperature regime in \cite{genovese2009mechanical,barra1,abarra,barra2}, in which case the corresponding partial differential equation is finite-dimensional. Later this method was revisited in~\cite{mourrat2020extending,mourrat2021nonconvex,mourrat2023free,HJbook,chen2022hamilton} to include the whole temperature range with the aim of finding a suitable limiting object to describe the free energy beyond the convex setting. What emerged from this line of work is the following conjecture.
\begin{conjecture}[\cite{mourrat2022parisi}] \label{conj:}
    Let $\psi(q) = \bar F_{1,\mu}(0,q)$. For every $t \geq 0$ we have
    \begin{equation}
        \lim_{N \to +\infty} \bar F_{N,\mu^{\otimes N}}(t,0) = f(t,0),
    \end{equation}
    where $f : \R_+ \times \mcl Q_2 \to \R$ is the solution of 
    \begin{equation} \label{e.hj}
        \begin{cases}
            \partial_t f(t,q) - \int_0^1 \xi(\partial_q f(t,q)(s))\dd s = 0 \\
            f(0,q) = \psi(q).
        \end{cases}
    \end{equation}
\end{conjecture}
An important result toward proving this conjecture is obtained in \cite{chen2025free}. The main result of \cite{chen2025free} is that if the free energy converges, then its limiting value is given by a certain functional $\mathscr J_{t,q}$ evaluated at a critical point. The functional $\mathscr J_{t,q}$ is strongly related to the Hamilton--Jacobi equation of the conjecture. This result has two important caveats that we will explain in the following. 

The first caveat is that the functional $\mathscr J_{t,q}$ in general has several critical points, so the main result of \cite{chen2025free} does not fully characterize the limiting value of the free energy.

The second caveat is that \cite{chen2025free} did not prove the convergence of the free energy. However, this missing convergence was provided  in~\cite{ho2026convergence}. Several key ingredients in the proof of this convergence are also used in the present paper.

The first key ingredient is to represent the free energy in terms of the solution of the so-called Parisi PDE. The Parisi PDE and the associated diffusion have a substantial history in scalar spin glasses ($D=1$). In particular, this PDE yields a stochastic-control representation which was used to establish the strict convexity of the Parisi functional~\cite{auffinger2015properties,auffinger2015parisi}; related dynamic-programming and stochastic-control approaches appear in~\cite{jagtob,chen2017variational}. A treatment of the Parisi PDE for multidimensional spins is given in~\cite[Section 6]{bovklim}. For vector spins, as considered here, the corresponding Parisi PDE and its diffusion and Hessian identities are developed in~\cite[Lemma 3.3]{chen2023parisi} and~\cite[Lemma 5.3]{chen2024simultaneous}. 

The second key ingredient is a log-concavity that was recently revealed in \cite{AronowLopatto2026}. This allows us to control fluctuations under a conditional log-concave law of Gaussian disorder and diffusion coordinates rather than directly under a Gibbs measure of the spins. 

Other ingredients include: the realization of the centered Gaussian Hamiltonian as $H_N(\sigma;\mathbf g)=h(\sigma)\cdot \mathbf g$ where $\mathbf g$ is a Gaussian vector with independent coordinates and $h(\sigma)$ is deterministic. This is an instance of covariance factorization; a related Hilbert-space representation for mixed $p$-spin models appears in~\cite[Section 1]{eldan2020chaos}. The Brascamp--Lieb variance inequality~\cite{brascampLieb1976}, used previously for long-range spin glasses~\cite{frohlichZegarlinski1989}, concentration of multioverlaps under log-concave Gibbs measures~\cite[Theorem 1]{barbierPanchenkoSaenz2022}, and variational bounds for perceptron models~\cite[Section 3.2 and Appendix B]{camilli2026perceptron} are also used.

Recently, a special case of Conjecture~\ref{conj:} was established in~\cite{chen2026ising} for multi-species models when the reference measure $\mu$ is the uniform measure on $\pm 1$. Those assumptions guarantee that the function $\psi$ is convex, which is the key ingredient to verifying  Conjecture~\ref{conj:} in that case. The corresponding result has also been established for  multi-species spherical spin glasses~\cite{chen2026spherical} using the same technique. Note that in general, the function $\psi$ is not convex so this technique cannot be used to solve the full conjecture. 

Theorem~\ref{thm:main} gives an unconditional variational formula for arbitrary $\mu$ through the constrained solutions $f^\nu$. Conjecture~\ref{conj:} asks whether this variational value at $q=0$ equals the single solution $f$ with initial condition $\psi$. That identification, though plausible, remains open in general. 


\subsection*{Proof strategy and organization} 
In Section~\ref{sec:enriched}, we give a precise definition of the enriched free energy $\bar F_{N,\lambda_N}(t,q)$ and verify its Lipschitz continuity with respect to $t,q$, and $\lambda_N$ (see Lemma~\ref{lem:regularity} and Lemma~\ref{lem:coupling}). 

For most of the manuscript, we assume that $\mu$ has finite support. Given a probability measure $\nu$ that is absolutely continuous with respect to $\mu$, we introduce $\lambda_N(\nu)$, which is the uniform measure on $\Sigma_N(\nu)$, the set of spin configurations $\sigma$ whose empirical measure is equal to $\nu$.

In Section~\ref{sec:residual} we prove the key result of our proof strategy: Proposition~\ref{prop:residual}. In the traditional approach of \cite{mourrat2020free,mourrat2021nonconvex} one wishes to verify that $\bar F_{N,\mu^{\otimes N}}$ approximately satisfies the partial differential equation \eqref{e.hj}. In this case, a major difficulty is that the corresponding error term is small but not \emph{uniformly} small in $(t,q)$. We verify in Lemma~\ref{lem:type-minimum} that as $N \to +\infty$ we have
\begin{equation} \label{eq:orga1}
    \bar F_{N,\mu^{\otimes N}} (t,q) = \min_{\nu} \left\{ \bar F_{N,\lambda_{N}(\nu)}(t,q) - \frac{1}{N} \log \mu^{\otimes N}(\Sigma_N(\nu)) \right\} + o(1).
\end{equation}
Proposition~\ref{prop:residual} then guarantees that, on the set of paths $q \in \mcl C^1([0,1],S^D_+)$ satisfying  $\dot q(s)\succeq\varepsilon I_D$, we have that $\bar F_{N,\lambda_N(\nu)}$ satisfies, up to an error term, the partial differential equation 
\begin{equation} \label{eq:orga2}
    \partial_t f^\nu(t,q) - \int_0^1 \xi(\partial_q f^\nu(t,q)(s))\dd s = 0.
\end{equation}
What is remarkable is that this error term goes to $0$ \emph{uniformly} in $(t,q)$ as $N \to +\infty$. This uniformity is proven using a new concentration result for a conditional version of the overlap (see Lemma~\ref{lem:diffusion-derivatives} and Lemma~\ref{lem:overlap-concentration}). The main result, Theorem~\ref{thm:main}, then directly follows (up to a few technicalities) from \eqref{eq:orga1}, Proposition~\ref{prop:residual}, and the fact that 
\begin{equation}
  \lim_{N \to +\infty}  \bar F_{N,\lambda_N(\nu)}(0,q) = \varphi^\nu(q),
\end{equation}
proven in Section~\ref{sec:finite-initial} (see Proposition~\ref{prop:finite-initial}). The few technicalities in question are addressed in Section~\ref{sec:constrained-HJ} and involve understanding how solutions of \eqref{eq:orga2} can be approximated by solutions of finite-dimensional partial differential equations. 

Finally, in Section~\ref{sec:ext-unit-ball} we remove the finiteness assumption on the support of $\mu$ by approximation. This uses the continuity of the free energy with respect to $\lambda_N$ proven in Section~\ref{sec:enriched}.

\subsection*{Acknowledgments} 
All three authors are deeply indebted to Jean-Christophe Mourrat for numerous insightful discussions on the Hamilton--Jacobi program and its ramifications. 
Moreover, FHH would like to thank Mourrat for his openness in sharing his ideas on multiple occasions, including the CRM--PIMS Summer School 2021 and the St. Flour Summer School 2024, from which FHH has benefited.

HBC acknowledges funding from the NYU Shanghai Start-Up Fund and support from the NYU--ECNU Institute of Mathematical Sciences at NYU Shanghai. FHH gratefully acknowledges support from Eliran Subag's ISF grant (No. 2055/21) and ERC grant (No. 101165541, Horizon Europe). VI acknowledges support from the FSMP postdoctoral program.

\subsection*{AI usage statement}

During the preparation of this manuscript, the authors used large language models to support the development of certain ideas and improve certain sections of text. All AI-generated inputs were critically reviewed, revised, and validated by the authors to ensure their accuracy, relevance, and alignment with the objectives of the study. 

\section{Enriched models and basic properties}\label{sec:enriched}

\subsection{Enriched free energies}
We now define the enrichment used throughout the paper. For $q\in\cQ_\infty$, let $q(1):=\lim_{s\uparrow1}q(s)$. Let $\cR$ be the continuous \textbf{Ruelle probability cascade} of~\cite[Section 4]{chen2025free}, defined on an infinite-dimensional separable Hilbert space with uniform overlap distribution. Its support $\supp\cR$ is contained in the unit sphere for almost every realization of $\cR$. If $\alpha$ and $\alpha'$ are independent samples from $\cR$, write $\alpha\wedge\alpha'$ for their Hilbert inner product. Under $\E\cR^{\otimes2}$, the random variable $\alpha\wedge\alpha'$ is uniformly distributed on $[0,1]$. Here and henceforth, we enlarge the probability space to include the independent random objects $H_N$ and $\cR$.
Conditionally on $\cR$, by further enlarging the probability space, let $W_N^q:\supp\cR\to\R^{D\times N}$ be a centered Gaussian process whose columns are independent and satisfy
\begin{equation*}
 \E\left[W^q_{N,\cdot,i}(\alpha)W^q_{N,\cdot,j}(\alpha')^\intercal\mid\cR\right]=\1_{\{i=j\}}q(\alpha\wedge\alpha'),\qquad i,j\in\{1,\dots,N\}.
\end{equation*}
The field $H_N$ is independent of $\cR$, and conditionally on $\cR$, the Gaussian processes $H_N$ and $W_N^q$ are independent. For $t\in\R_+$ and $q\in\cQ_\infty$, define $H_N^{t,q}:\B^N\times\supp\cR\to\R$ by
\begin{equation*}
 H_N^{t,q}(\sigma,\alpha):=\sqrt{2t}\,H_N(\sigma)-tN\xi\bigl(\overlap_N(\sigma,\sigma)\bigr)+\sqrt2\,W_N^q(\alpha)\cdot\sigma-q(1)\cdot\sigma\sigma^\intercal.
\end{equation*}
For a nonzero finite measure $\lambda_N$ on $\B^N$, define $\bar F_{N,\lambda_N}:\R_+\times\cQ_\infty\to\R$ by
\begin{equation}\label{eq:free-energy-general}
 \bar F_{N,\lambda_N}(t,q):=-\frac1N\E\log\iint\exp\bigl(H_N^{t,q}(\sigma,\alpha)\bigr)\lambda_N(\mathrm d\sigma)\cR(\mathrm d\alpha).
\end{equation}
For fixed disorder and cascade, write $\langle\cdot\rangle$ for expectation under the probability measure proportional to $\exp(H_N^{t,q}(\sigma,\alpha))\lambda_N(\mathrm d\sigma)\cR(\mathrm d\alpha)$. Primed or numbered replicas are conditionally independent samples from this measure. Lemma~\ref{lem:regularity} shows that~\eqref{eq:free-energy-general} extends uniquely from $\cQ_\infty$ to $\cQ_1$ by $L^1$ continuity, and we use the same notation for this extension. For a step path $q$, the same enrichment can equivalently be constructed from the finite-level Ruelle probability cascade associated with the jump points of $q$; see~\cite[Section 4]{chen2025free}. The continuous cascade formulation above is convenient because it applies directly to every bounded path.

We equip $L^2([0,1];S^D)$ with the inner product $\ip p q_{L^2}:=\int_0^1p(s)\cdot q(s)\dd s$. The dual cone of $\cQ_2$ is
\begin{equation}\label{eq:dual-cone}
 \cQ_2^*:=\left\{r\in L^2([0,1];S^D):\ip r q_{L^2}\ge0\ \text{for every }q\in\cQ_2\right\}.
\end{equation}

\subsection{Gaussian comparisons and regularity}
We first specify the notion of differentiability used below. Let $E$ be a Hilbert space, let $G\subset E$, and let $g:G\to\R$. Following~\cite[Definition 2.1]{chen2025free}, we say that $g$ is \textbf{Fr\'echet differentiable relative to $G$} at $x\in G$ if there exists a unique $y\in E$ such that
\begin{equation*}
 \lim_{\substack{x'\in G,\ x'\to x\\x'\ne x}}
 \frac{|g(x')-g(x)-\ip{y}{x'-x}_E|}{|x'-x|_E}=0.
\end{equation*}
We call $y$ the derivative of $g$ at $x$. For $G=\R_+\times\cQ_2$, the ambient Hilbert space is $E:=\R\times L^2([0,1];S^D)$. We write the two components of the derivative as $\partial_t g$ and $\partial_q g$.
For a subset $G\subset L^2([0,1];S^D)$, we say that a function $g:G\to\R$ is \textbf{$\cQ_2^*$-increasing} if
\begin{equation*}
 q-q'\in\cQ_2^*\quad\Longrightarrow\quad g(q)\ge g(q'),
 \qquad q,q'\in G;
\end{equation*}
cf.~\cite[Equation (3.19)]{chen2025free}.
The following lemma can be extracted from \cite[Propositions 3.1, 3.6, and 5.1]{chen2025free}. 
We adopt the following abbreviations for the derivatives of $\xi$:
\begin{equation}\label{e.M_i=}
 M_0:=\sup_{a\in \R^{D\times D}:\:|a|\le1}|\xi(a)|,
 \quad
 M_1:=\sup_{a\in \R^{D\times D}:\:|a|\le1}|\nabla\xi(a)|,
 \quad \text{and}\quad
M_2:=\sup_{a\in\R^{D\times D}:\:|a|\le1}\norm{\nabla^2\xi(a)}_{\mathrm{op}}.
\end{equation}
\begin{lemma}\label{lem:regularity}
For every nonzero finite measure $\lambda_N$ supported on $\B^N$ and all $(t,q),(t',q')\in\R_+\times\cQ_1$, we have
\begin{equation}\label{eq:regularity}
 |\bar F_{N,\lambda_N}(t,q)-\bar F_{N,\lambda_N}(t',q')|
 \le M_0|t-t'|+|q-q'|_{L^1}.
\end{equation}
The restriction of $\bar F_{N,\lambda_N}$ to $\R_+\times\cQ_2$ is Fr\'echet differentiable relative to its domain. 
The derivative $\partial_q\bar F_{N,\lambda_N}(t,q)$ belongs to $\mcl Q_\infty$ viewed as a path
\begin{align}\label{e.path_view_of_dqF}
    [0,1)\ni s \quad\mapsto \quad \partial_q\bar F_{N,\lambda_N}(t,q,s) \in S^D_+
\end{align}
and satisfies $\Ll|\partial_q\bar F_{N,\lambda_N}(t,q)\Rr|_{L^\infty}\leq 1$.
For every $t\in\R_+$, $q\in\cQ_\infty$, and $\kappa\in L^2([0,1];S^D)$, we have
\begin{equation}\label{eq:replica-derivatives}
  \partial_t\bar F_{N,\lambda_N}(t,q)
  =\E\langle\xi(\overlap_N(\sigma,\sigma'))\rangle \quad\text{and}\quad
  \ip\kappa{\partial_q\bar F_{N,\lambda_N}(t,q)}_{L^2}
  =\E\langle\kappa(\alpha\wedge\alpha')\cdot \overlap_N(\sigma,\sigma')\rangle.
\end{equation}
For fixed $N\in\N$ and $t\in\R_+$, we have
\begin{equation}\label{eq:gradient-continuity}
 |\partial_q\bar F_{N,\lambda_N}(t,q)-\partial_q\bar F_{N,\lambda_N}(t,q')|_{L^1}
 \le16N|q-q'|_{L^1},
 \qquad q,q'\in\cQ_2.
\end{equation}
For fixed $N\in\N$ and $t\in\R_+$, the function $q\mapsto\bar F_{N,\lambda_N}(t,q)$ is $\cQ_2^*$-increasing on $\cQ_2$.
\end{lemma}

\begin{lemma}\label{lem:initial-regularity}
For every $\nu\in\cP(\supp\mu)$, we have $\varphi^\nu(0)=0$ and
\begin{equation}\label{eq:initial-regularity}
 0\le\varphi^\nu(q)\le|q|_{L^1},
 \qquad
 |\varphi^\nu(q)-\varphi^\nu(q')|\le|q-q'|_{L^1},
 \qquad q,q'\in\cQ_1.
\end{equation}
Its restriction to $\cQ_2$ is $\cQ_2^*$-increasing.
\end{lemma}
\begin{proof}
At $q=0$, Jensen's inequality gives
\begin{equation*}
 \psi_{\nu,h}(0)+\int h\,\mathrm d\nu
 =-\log\int e^h\,\mathrm d\nu+\int h\,\mathrm d\nu\le0.
\end{equation*}
Equality holds for $h=0$, so $\varphi^\nu(0)=0$. By Lemma~\ref{lem:regularity}, each function $q\mapsto\psi_{\nu,h}(q)+\int h\,\mathrm d\nu$ is $1$-Lipschitz in $L^1$ and $\cQ_2^*$-increasing on $\cQ_2$. Taking the supremum in~\eqref{eq:initial-definition} preserves these properties and gives $\varphi^\nu(q)\le|q|_{L^1}$. 
Since $q\in\cQ_2$ implies $q\in\cQ_2^*$, monotonicity gives $\varphi^\nu(q)\ge\varphi^\nu(0)=0$ there. 
Approximation by bounded paths extends this lower bound to $\cQ_1$.
\end{proof}

\begin{lemma}\label{lem:coupling}
Let $\lambda_N$ and $\lambda_N'$ be probability measures on $\B^N$. Suppose that they admit a coupling supported on pairs $(\sigma,\tau)$ satisfying $N^{-1}\sum_{i=1}^N|\sigma_{\cdot,i}-\tau_{\cdot,i}|\le a$. Then, for every $t\in\R_+$ and $q\in\cQ_\infty$, we have
\begin{equation}\label{eq:coupling-bound}
 |\bar F_{N,\lambda_N}(t,q)-\bar F_{N,\lambda_N'}(t,q)|
 \le2a\bigl(tM_1+|q|_{L^\infty}\bigr).
\end{equation}
\end{lemma}
\begin{proof}
Let $\Pi_N$ be a coupling of $\lambda_N$ and $\lambda_N'$ satisfying the assumption of the lemma. We use $x=(\sigma,\tau,\alpha)$ and $x'=(\sigma',\tau',\alpha')$ for two points in $\B^N\times\B^N\times\supp\cR$. Conditionally on $\cR$, let $X_0$ and $X_1$ be independent centered Gaussian processes indexed by this space with covariance kernels
\begin{equation}\label{eq:coupling-covariances}
 \begin{aligned}
  \Cov_0(x,x')&:=2tN\xi\bigl(\overlap_N(\sigma,\sigma')\bigr)+2Nq(\alpha\wedge\alpha')\cdot \overlap_N(\sigma,\sigma'),\\
  \Cov_1(x,x')&:=2tN\xi\bigl(\overlap_N(\tau,\tau')\bigr)+2Nq(\alpha\wedge\alpha')\cdot \overlap_N(\tau,\tau').
 \end{aligned}
\end{equation}
These are the covariance kernels of the centered Gaussian parts of the enriched Hamiltonians evaluated respectively at $\sigma$ and $\tau$. In particular, we have
\begin{align*}
 \frac12\Cov_0(x,x)=tN\xi\bigl(\overlap_N(\sigma,\sigma)\bigr)+q(1)\cdot\sigma\sigma^\intercal \qquad\text{and}\qquad 
 \frac12\Cov_1(x,x)=tN\xi\bigl(\overlap_N(\tau,\tau)\bigr)+q(1)\cdot\tau\tau^\intercal.
\end{align*}
Thus the diagonal terms in the enriched Hamiltonian are exactly $-\Cov_0(x,x)/2$ and $-\Cov_1(x,x)/2$ at the two endpoints of the interpolation.

For $\theta\in[0,1]$, define $\widetilde X_\theta:=\sqrt\theta\,X_0+\sqrt{1-\theta}\,X_1$, $\widetilde{\Cov}_\theta:=\theta \Cov_0+(1-\theta)\Cov_1$, and define $\phi:[0,1]\to\R$ by
\begin{equation*}
 \phi(\theta):=-\frac1N\E\log\iint
 \exp\left(\widetilde X_\theta(x)-\frac12\widetilde{\Cov}_\theta(x,x)\right)
 \Pi_N(\mathrm d\sigma,\mathrm d\tau)\cR(\mathrm d\alpha).
\end{equation*}
Since the first and second marginals of $\Pi_N$ are $\lambda_N$ and $\lambda_N'$, respectively, the endpoint values are
\begin{equation*}
 \phi(1)=\bar F_{N,\lambda_N}(t,q)
 \quad\text{and}\quad
 \phi(0)=\bar F_{N,\lambda_N'}(t,q).
\end{equation*}
Let $\langle\cdot\rangle_\theta$ denote the Gibbs expectation associated with $\phi(\theta)$, and let $x^1,x^2$ be two replicas. Gaussian integration by parts, together with the derivative of the compensation $-\widetilde{\Cov}_\theta(x,x)/2$, gives
\begin{equation}\label{eq:coupling-interpolation-derivative}
 \phi'(\theta)=\frac1{2N}\E\left\langle
 (\Cov_0-\Cov_1)(x^1,x^2)
 \right\rangle_\theta.
\end{equation}

We next bound the covariance difference in~\eqref{eq:coupling-interpolation-derivative}. If $(\sigma,\tau)$ and $(\sigma',\tau')$ belong to the support of $\Pi_N$, then
\begin{align*}
 &|\overlap_N(\sigma,\sigma')-\overlap_N(\tau,\tau')|\le\frac1N\sum_{i=1}^N
 \left|\sigma_{\cdot,i}(\sigma'_{\cdot,i})^\intercal-\tau_{\cdot,i}(\tau'_{\cdot,i})^\intercal\right|\le\frac1N\sum_{i=1}^N
 \left(|\sigma_{\cdot,i}-\tau_{\cdot,i}|+|\sigma'_{\cdot,i}-\tau'_{\cdot,i}|\right)
 \le2a.
\end{align*}
Every overlap in this display belongs to the unit ball of $\R^{D\times D}$. Hence, the definition of $M_1$ gives
\begin{equation*}
 \left|\xi\bigl(\overlap_N(\sigma,\sigma')\bigr)-\xi\bigl(\overlap_N(\tau,\tau')\bigr)\right|
 \le2M_1a\quad\text{and}\quad \left|q(u)\cdot\bigl(\overlap_N(\sigma,\sigma')-\overlap_N(\tau,\tau')\bigr)\right|
 \le2a|q|_{L^\infty}.
\end{equation*}
Using~\eqref{eq:coupling-covariances}, we therefore obtain $|(\Cov_0-\Cov_1)(x,x')|
 \le4Na\bigl(tM_1+|q|_{L^\infty}\bigr)$.
Substituting this estimate into~\eqref{eq:coupling-interpolation-derivative} yields $|\phi'(\theta)|\le2a\bigl(tM_1+|q|_{L^\infty}\bigr)$ for every $\theta\in[0,1]$.
Integrating over $\theta\in[0,1]$ proves~\eqref{eq:coupling-bound}.
\end{proof}

Assume that $\mu$ has finite support and write $\supp\mu=\{a_1,\ldots,a_m\}$, with $a_1,\ldots,a_m$ distinct. For $\nu\in\cP(\supp\mu)$, write $\nu_j:=\nu(\{a_j\})$ for $j\in\{1,\ldots,m\}$. For $N\in\N$, define $\Delta_{m,N}\subset\cP(\supp\mu)$ and $\Sigma_N:\Delta_{m,N}\to2^{(\supp\mu)^N}$ by
\begin{align}
 \Delta_{m,N}&:=\left\{\nu\in\cP(\supp\mu):N\nu_j\in\{0,\ldots,N\},\ \forall j\in\{1,\ldots,m\}\right\},\nonumber\\
 \Sigma_N(\nu)&:=\left\{\sigma\in (\supp\mu)^N:\frac1N\#\{i\in\{1,\ldots,N\}:\sigma_{\cdot,i}=a_j\}=\nu_j,\ \forall j\in\{1,\ldots,m\}\right\}.
 \label{eq:type-set}
\end{align}
For each $\nu \in \Delta_{m,N}$, let $\lambda_N(\nu)$ be the uniform probability measure on $\Sigma_N(\nu)$.
For these admissible empirical measures, write
\begin{equation}\label{eq:normalized-constrained-notation}
 \bar F_N^\nu:=\bar F_{N,\lambda_N(\nu)}.
\end{equation}
Define $\mu_j:=\mu(\{a_j\})>0$ for $j\in\{1,\ldots,m\}$. Thus $\mu=\sum_{j=1}^m\mu_j\delta_{a_j}$. For $\nu\in\Delta_{m,N}$, define $\bar F_{N,\mu^{\otimes N},\Sigma_N(\nu)}:\R_+\times\cQ_\infty\to\R$ by
\begin{equation}\label{eq:constrained-free-energy}
 \bar F_{N,\mu^{\otimes N},\Sigma_N(\nu)}(t,q):=-\frac1N\E\log
 \int_{\Sigma_N(\nu)}\int_{\supp\cR}
 \exp\bigl(H_N^{t,q}(\sigma,\alpha)\bigr)
 \cR(\mathrm d\alpha)\mu^{\otimes N}(\mathrm d\sigma).
\end{equation}
By Lemma~\ref{lem:regularity}, $\bar F_{N,\mu^{\otimes N},\Sigma_N(\nu)}$ has a unique $L^1$-Lipschitz extension to $\R_+\times\cQ_1$; we use the same notation for this extension.

For every $\nu\in\Delta_{m,N}$, $\mu^{\otimes N}$ is constant on $\Sigma_N(\nu)$ and we have
\begin{equation}\label{eq:type-mass}
 |\Sigma_N(\nu)|=\frac{N!}{\prod_{j=1}^m(N\nu_j)!}
 \quad\text{and}\quad
 \mu^{\otimes N}(\Sigma_N(\nu))
 =\frac{N!}{\prod_{j=1}^m(N\nu_j)!}\prod_{j=1}^m\mu_j^{N\nu_j}.
\end{equation}
In particular, we have
\begin{equation}\label{eq:restricted-reference-normalization}
 \1_{\Sigma_N(\nu)}\mu^{\otimes N}=\mu^{\otimes N}(\Sigma_N(\nu))\lambda_N(\nu).
\end{equation}
Equations~\eqref{eq:constrained-free-energy} and~\eqref{eq:restricted-reference-normalization} give
\begin{equation}\label{eq:normalization}
 \bar F_{N,\mu^{\otimes N},\Sigma_N(\nu)}(t,q)
 =\bar F_N^\nu(t,q)-\frac1N\log\mu^{\otimes N}(\Sigma_N(\nu)).
\end{equation}
\begin{lemma}\label{lem:type-probability}
For every $N\in\N$ and $\nu\in\Delta_{m,N}$,
\begin{equation}\label{eq:stirling}
 0\le-\frac1N\log\mu^{\otimes N}(\Sigma_N(\nu))-\Ent(\nu\mid\mu)
 \le\frac{m\log(N+1)}N.
\end{equation}
\end{lemma}
\begin{proof} 
We have the exact identity 
\begin{equation}\label{eq:type-probability-identity}
 \mu^{\otimes N}(\Sigma_N(\nu))
 =e^{-N\Ent(\nu\mid\mu)}\nu^{\otimes N}(\Sigma_N(\nu)).
\end{equation}
To conclude, it suffices to show that $\nu^{\otimes N}(\Sigma_N(\nu)) \geq (N+1)^{-m}$. First observe that by multinomial factorization, we have 
\begin{equation}
    \sum_{\ell} \frac{N!}{\prod_{j=1}^m \ell_j!} \prod_{j  =1}^m \nu_j^{\ell_j}= (\nu_1+\dots+\nu_m)^N = 1,
\end{equation}
where the sum is taken over $\ell \in \{0,\dots,N\}^m$ such that $\sum_j \ell_j = N$. Also observe that the map $\ell \mapsto (N\nu_j)^\ell/\ell!$ is maximized over $\N$ at $\ell = N \nu_j$. Hennce, each term in the sum of the previous display is $\leq \nu^{\otimes N}(\Sigma_N(\nu))$. Since there is less than $(N+1)^m$ terms in the sum, we obtain 
\begin{equation}
    (N+1)^m \nu^{\otimes N}(\Sigma_N(\nu)) \geq \sum_{\ell} \frac{N!}{\prod_{j=1}^m \ell_j!} \prod_{j  =1}^m \nu_j^{\ell_j}= (\nu_1+\dots\nu_m)^N = 1,
\end{equation}
thus the result is proven.
\end{proof}


Let $\nu \in \Delta_{m,N}$, assume that there exists $x_1,\dots,x_N \in \B$ such that $\nu=N^{-1}\sum_{i=1}^N\delta_{x_i}$. Then, we have 
\begin{equation}
    \Sigma_N(\nu)  = \{ (x_{\pi(i)})_{1 \leq i \leq N} : \, \pi \in \mfk S_N \}.
\end{equation}
In particular, if $\Pi$ is a uniform random variable on the symmetric group $\mathfrak S_N$ and $\sigma_{\cdot,i} = x_{\Pi(i)}$, then $\sigma$ is a random variable with law $\lambda_N(\nu)$.

\begin{lemma}\label{lem:type-continuity}
If $\nu=N^{-1}\sum_{i=1}^N\delta_{x_i}$ and $\nu'=N^{-1}\sum_{i=1}^N\delta_{y_i}$ for $x_i,y_i\in\B$, then for every $(t,q) \in \R_+ \times \mcl Q_\infty$,
\begin{equation}\label{eq:type-continuity}
 |\bar F_N^\nu(t,q)-\bar F_N^{\nu'}(t,q)|
 \le2\bigl(tM_1+|q|_{L^\infty}\bigr)W_1(\nu,\nu').
\end{equation}
Here $W_1$ is the Wasserstein distance on $\cP(\B)$ with cost $|a-b|$.
\end{lemma}
\begin{proof}
After relabeling $(y_i)_{i=1}^N$ (by the Birkhoff--von Neumann theorem), we may assume that
\begin{equation*}
 W_1(\nu,\nu')=\frac1N\sum_{i=1}^N|x_i-y_i|.
\end{equation*}
Let $\Pi$ be uniform on the symmetric group $\mathfrak S_N$, and define $\sigma,\tau\in\B^N$ by
\begin{equation*}
 \sigma_{\cdot,i}:=x_{\Pi(i)}
 \quad\text{and}\quad
 \tau_{\cdot,i}:=y_{\Pi(i)},
 \qquad i\in\{1,\ldots,N\}.
\end{equation*}
The law of $\sigma$ is $\lambda_N(\nu)$, and the law of $\tau$ is $\lambda_N(\nu')$. Moreover, every pair in the support of this coupling satisfies
\begin{equation*}
 \frac1N\sum_{i=1}^N|\sigma_{\cdot,i}-\tau_{\cdot,i}|
 =\frac1N\sum_{i=1}^N|x_i-y_i|
 =W_1(\nu,\nu').
\end{equation*}
Lemma~\ref{lem:coupling}, with $a=W_1(\nu,\nu')$, gives~\eqref{eq:type-continuity}.
\end{proof}

\section{An approximate Hamilton--Jacobi equation for constrained free energy}\label{sec:residual}

The goal of this section is to prove the following approximate Hamilton--Jacobi equation for the constrained free energy.

\begin{proposition}\label{prop:residual}
Assume $\supp\mu=\{a_1,\ldots,a_m\}$.
There is a constant $C=C_{D,M_2}<\infty$, where $M_2$ was defined in~\eqref{e.M_i=}, such that
for every $N\in\N$, $\nu\in\Delta_{m,N}$, $t>0$, $\varepsilon>0$, and $q\in C^1([0,1];S^D_+)$ satisfying $\dot q(s)\succeq\varepsilon I_D$ for all $s\in[0,1]$,
\begin{equation}\label{eq:residual}
 \left|\partial_t\bar F_N^\nu(t,q)
 -\int_0^1\xi\bigl(\partial_q\bar F_N^\nu(t,q,s)\bigr)\dd s\right|
 \le C(N\varepsilon)^{-1/3}.
\end{equation}
\end{proposition}

As explained in the introduction, the proof uses ingredients similar to those of the proof of~\cite[Proposition 3.9]{ho2026convergence}. 

Throughout the first two subsections, we fix $N,\nu,t,\varepsilon,q$ satisfying the assumptions of Proposition~\ref{prop:residual}.

\subsection{Parisi representation and derivative identities}

Let $\mathrm{Cov}_N$ be the covariance matrix on $(\supp\mu)^N$, with entries $\mathrm{Cov}_N(\sigma,\tau):=N\xi(\overlap_N(\sigma,\tau))$. By~\eqref{eq:covariance}, $\mathrm{Cov}_N\succeq0$. 
For the standard basis $(\mathbf e_\sigma)_{\sigma\in (\supp\mu)^N}$ of $\R^{(\supp\mu)^N}$, 
put $\mathbf h_N(\sigma):=\mathrm{Cov}_N^{1/2}\mathbf e_\sigma$. 
If $\bg$ is a standard Gaussian vector in $\R^{(\supp\mu)^N}$, then
\begin{equation}\label{eq:orbit-realization}
 H_N(\sigma;\bg):=\mathbf h_N(\sigma)\cdot \bg,
 \qquad \mathbf h_N(\sigma)\cdot\mathbf h_N(\tau)=N\xi(\overlap_N(\sigma,\tau)).
\end{equation}
We use this realization to differentiate in the disorder and to exploit permutation symmetry.

For $\pi\in\mathfrak S_N$, write $\sigma^\pi_{\cdot,i}:=\sigma_{\cdot,\pi(i)}$ and let $Q_\pi \mathbf e_\sigma:=\mathbf e_{\sigma^\pi}$. 
Since $\mathrm{Cov}_N$ and $\mathrm{Cov}_N^{1/2}$ commute with $Q_\pi$,
\begin{equation}\label{eq:orbit-equivariance}
 H_N(\sigma^\pi;Q_\pi \bg)=H_N(\sigma;\bg).
\end{equation}

For $\bg\in\R^{(\supp\mu)^N}$, define a positive finite measure $\Lambda_{\bg}$ on $\Sigma_N(\nu)$ by
\begin{equation*}
 \Lambda_{\bg}(\mathrm d\sigma):=
 \exp\left(\sqrt{2t}\,H_N(\sigma;\bg)-tN\xi(\overlap_N(\sigma,\sigma))\right)
 \lambda_N(\nu)(\mathrm d\sigma).
\end{equation*}
Let $\Phi:[0,1]\times\R^{D\times N}\times\R^{(\supp\mu)^N}\to\R$ solve the following Parisi PDE
\begin{align}
    \label{eq:Parisi-PDE}
    \begin{cases}
 -\partial_s\Phi(s,\bx;\bg)
 = 
 \Bigl(
 \sum_{i=1}^N\dot q(s)\cdot
 \nabla^2_{\bx_{\cdot,i}}\Phi
+s\dot q(s)\cdot\nabla_{\bx}\Phi
 (\nabla_{\bx}\Phi)^\intercal
 \Bigr)
 (s,\bx;\bg),
& (s,\bx;\bg)\in [0,1]\times\R^{D\times N}\times\R^{(\supp\mu)^N}
 \\
 \Phi(1,\bx;\bg)=\log\int\exp\left(\bx\cdot\sigma-q(1)\cdot\sigma\sigma^\intercal\right)\Lambda_{\bg}(\mathrm d\sigma),
& (\bx;\bg)\in \R^{D\times N}\times\R^{(\supp\mu)^N}.
\end{cases}
\end{align}
Let $\bz_0$ be standard Gaussian in $\R^{D\times N}$, and let $B:[0,1]\to\R^{D\times N}$ be a standard Brownian motion, 
with $\bg,\bz_0,B$ mutually independent and independent of the cascade.

\begin{lemma}\label{lem:Parisi-diffusion}
The equation~\eqref{eq:Parisi-PDE} has a unique classical solution $\Phi$ with bounded spatial derivatives of positive order, and $|\nabla_{\bx_{\cdot,i}}\Phi|\le1$ for every $i\in\{1,\dots,N\}$. Moreover, we have
\begin{equation}\label{e.F_N=-1/NEPhi}
 \bar F_N^\nu(t,q)=-\frac1N\E\Phi(0,\sqrt{2q(0)}\bz_0;\bg).
\end{equation}
The stochastic differential equation
\begin{equation}\label{eq:diffusion}
 \mathrm d\bX_s=2s\dot q(s)\nabla_{\bx}\Phi(s,\bX_s;\bg)\dd s+\sqrt{2\dot q(s)}\,\mathrm dB_s,
 \quad\text{and}\quad \bX_0=\sqrt{2q(0)}\bz_0
\end{equation}
admits a unique strong solution on $[0,1]$.
\end{lemma}
\begin{proof}
In column-stacked coordinates,~\eqref{eq:Parisi-PDE} has covariance path $I_N\otimes q$, coefficient $\alpha(s)=s$, and terminal measure $\Lambda_{\bg}$. Existence and regularity follow from~\cite[Proposition 4.7(1), (3), and (4)]{chen2024simultaneous}. Its cascade representation~\cite[Lemma 4.5, (4.10)]{chen2024simultaneous} gives~\eqref{e.F_N=-1/NEPhi} after averaging over $(\bg,\bz_0)$.
Differentiating in $\bx_{\cdot,i}$ gives a linear backward equation whose terminal gradient lies in the unit ball. The maximum principle gives the column bound. The bounded Hessian makes the drift of~\eqref{eq:diffusion} globally Lipschitz, giving a unique strong solution; see also~\cite[Proposition 3.2]{chen2023parisi}.
The finite terminal sum and the same regularity estimates justify the parameter differentiations below: derivatives in $(\bx,\bg)$ and in bounded terminal perturbations are bounded, while the $t$ derivative has at most linear growth in $\bg$ locally on $t>0$.
\end{proof}

Define the filtration $\mathcal G_s:=\sigma\bigl(\bg,\bz_0,(B_r)_{0\le r\le s}\bigr)$ for $s\in[0,1]$.
For all $s\in [0,1]$, let $\sigma_s$ have the conditional law
\begin{equation*}
 \P(\sigma_s=\sigma\mid\mathcal G_s)
 =
\Lambda_{\bg}(\{\sigma\})\,
\mathbb E\!\left[
\exp(\bX_1\cdot\sigma-q(1)\cdot\sigma\sigma^\intercal
-\Phi(1,\bX_1;\bg))
\,\middle|\,\mathcal G_s
\right].
\end{equation*}
Here $\sigma_1$ denotes the terminal spin.

Let $\operatorname{vec}$ stack the columns of a $D$-by-$N$ matrix. Spatial Hessians are taken in these column-stacked coordinates.

\begin{lemma}\label{lem:Parisi-Hessian}
The following statements hold.
\begin{enumerate}
\item 
The function $\Phi(s,\cdot\,;\cdot)$ is jointly convex in $(\bx,\bg)$. 
\item The process $\E[\sigma_1\mid\mathcal G_s]$ is a bounded martingale, and for every $s\in[0,1]$,
\begin{align*}
 \E[\sigma_1\mid\mathcal G_s]=\nabla_{\bx}\Phi(s,\bX_s;\bg).
\end{align*}
\item We have
\begin{align*}
 \nabla_{\bx}^2\Phi(s,\bX_s;\bg)=s\mathrm{Cov}(\operatorname{vec}(\sigma_1)\mid\mathcal G_s)+\E\left[\int_s^1\mathrm{Cov}(\operatorname{vec}(\sigma_1)\mid\mathcal G_r)\dd r\,\middle|\,\mathcal G_s\right].
\end{align*}
\item We have
\begin{equation*}
    \int_0^1\E\bigl|\nabla_{\bx}^2\Phi(s,\bX_s;\bg)\bigr|^2\dd s\le\frac{N}{2\varepsilon}.
\end{equation*}
\end{enumerate}
\end{lemma}
\begin{proof}
For a bounded $v:\Sigma_N(\nu)\to\R$, let $\Phi^y$ have terminal measure $e^{yv}\Lambda_{\bg}$. All parameter derivatives below are evaluated at $y=0$ along the unperturbed diffusion. Differentiating~\eqref{eq:Parisi-PDE} once in $y$ and applying It\^o's formula show that the first derivative is a martingale with terminal value $\E[v(\sigma_1)\mid\mathcal G_1]$. Hence
\begin{equation*}
 \left.\partial_y\Phi^y(s,\bX_s;\bg)\right|_{y=0}
 =\E[v(\sigma_1)\mid\mathcal G_s].
\end{equation*}
Differentiating twice shows that the second derivative has drift $-s\,\mathrm d[\E[v(\sigma_1)\mid\mathcal G_\cdot]]_s$, whereas the conditional variance $\Var(v(\sigma_1)\mid\mathcal G_s)$ has drift $-\mathrm d[\E[v(\sigma_1)\mid\mathcal G_\cdot]]_s$. Subtracting $s\Var(v(\sigma_1)\mid\mathcal G_s)$ from the second derivative therefore gives a process with drift $-\Var(v(\sigma_1)\mid\mathcal G_s)\dd s$ and zero terminal value. Integration and conditioning give
\begin{equation}\label{eq:scalar-Hessian-identity}
 \begin{aligned}
 \left.\partial_y^2\Phi^y(s,\bX_s;\bg)\right|_{y=0}
 =s\Var(v(\sigma_1)\mid\mathcal G_s)
 +\E\left[\int_s^1\Var(v(\sigma_1)\mid\mathcal G_r)\dd r\,\middle|\,\mathcal G_s\right].
 \end{aligned}
\end{equation}
We now prove the four items.

\smallskip
\noindent\emph{Proof of (1):}
For $a\in\R^{D\times N}$ and $b\in\R^{(\supp\mu)^N}$, take
\begin{equation*}
 v(\sigma)=a\cdot\sigma+\sqrt{2t}\,b\cdot\mathbf h_N(\sigma).
\end{equation*}
Uniqueness gives $\Phi^y(s,\bx;\bg)=\Phi(s,\bx+ya;\bg+yb)$, so $\left.\partial_y^2\Phi^y(s,\bX_s;\bg)\right|_{y=0}$ is the second directional derivative in direction $(a,b)$. By~\eqref{eq:scalar-Hessian-identity}, it is nonnegative. Restarting the diffusion from any deterministic $(s,\bx)$, with $\bg$ fixed, proves joint convexity at every $(s,\bx,\bg)$.

\smallskip
\noindent\emph{Proof of (2):}
For $v(\sigma)=a\cdot\sigma$ with $a\in\R^{D\times N}$, uniqueness gives $\Phi^{\mathbf y}(s,\bx;\bg)=\Phi(s,\bx+\mathbf ya;\bg)$. Thus
\begin{equation*}
 a\cdot\nabla_{\bx}\Phi(s,\bX_s;\bg)=\left.\partial_{\mathbf y}\Phi^{\mathbf y}(s,\bX_s;\bg)\right|_{\mathbf y=0}
 =\E[a\cdot\sigma_1\mid\mathcal G_s].
\end{equation*}
Since $a$ is arbitrary, this proves the gradient identity. The conditional-expectation process is a martingale and is bounded because $|\sigma_1|^2\le N$.

\smallskip
\noindent\emph{Proof of (3):}
For the same spatial perturbation, the second parameter derivative is the Hessian quadratic form in direction $a$. Applying~\eqref{eq:scalar-Hessian-identity} with $v(\sigma)=a\cdot\sigma$ gives
\begin{align*}
 \operatorname{vec}(a)^\intercal \nabla_{\bx}^2\Phi(s,\bX_s;\bg)\operatorname{vec}(a)
 =s\Var(a\cdot\sigma_1\mid\mathcal G_s)
 +\E\left[\int_s^1\Var(a\cdot\sigma_1\mid\mathcal G_r)\dd r\,\middle|\,\mathcal G_s\right].
\end{align*}
Since $a$ is arbitrary, this proves the conditional covariance representation of $\nabla_{\bx}^2\Phi(s,\bX_s;\bg)$.

\smallskip
\noindent\emph{Proof of (4):}
Differentiating~\eqref{eq:Parisi-PDE} once in $\bx$, applying It\^o's formula, and using Lemma~\ref{lem:Parisi-Hessian}(2) give
\begin{equation}\label{eq:Parisi-gradient-Ito}
 \mathrm d\operatorname{vec}(\E[\sigma_1\mid\mathcal G_s])
 =\nabla_{\bx}^2\Phi(s,\bX_s;\bg)\sqrt{2(I_N\otimes\dot q(s))}\,\mathrm d\operatorname{vec}(B_s);
\end{equation}
see~\cite[Lemma 3.3]{chen2023parisi} and~\cite[Lemma 5.3]{chen2024simultaneous}. Since $(I_N\otimes\dot q(s))\succeq\varepsilon I_{DN}$, It\^o's isometry gives
\begin{equation*}
 2\varepsilon\int_0^1\E\bigl|\nabla_{\bx}^2\Phi(s,\bX_s;\bg)\bigr|^2\dd s
 \le\E|\E[\sigma_1\mid\mathcal G_1]|^2-\E|\E[\sigma_1\mid\mathcal G_0]|^2\le\E|\sigma_1|^2\le N.
\end{equation*}
Dividing by $2\varepsilon$ proves the estimate.
\end{proof}

We now provide the derivative identities.

\begin{lemma}\label{lem:diffusion-derivatives}
Let $\sigma_s^1,\sigma_s^2$ be conditionally independent copies of $\sigma_s$ conditioned on $\mathcal G_s$. 
Let $\overlap_s:=N^{-1}\sigma_s^1(\sigma_s^2)^\intercal$. 
Then,
\begin{equation}\label{eq:diffusion-derivatives}
 \partial_q\bar F_N^\nu(t,q,s)=\E\overlap_s
 \quad\text{and}\quad
 \partial_t\bar F_N^\nu(t,q)=\int_0^1\E\xi(\overlap_s)\dd s.
\end{equation}
\end{lemma}
\begin{proof}

We first prove the time identity. Differentiating~\eqref{eq:Parisi-PDE} in $t$, applying the diffusion representation, and using~\eqref{eq:orbit-realization} and~\eqref{e.F_N=-1/NEPhi} give
\begin{equation*}
 \partial_t\bar F_N^\nu(t,q)
 =\frac1N\E\left[|\mathbf h_N(\sigma_1)|^2-\frac{\bg\cdot\mathbf h_N(\sigma_1)}{\sqrt{2t}}\right].
\end{equation*}
Since $\bX_0$ is independent of $\bg$, Gaussian integration by parts and~\eqref{eq:scalar-Hessian-identity} in the proof of Lemma~\ref{lem:Parisi-Hessian} at $s=0$, with $v=\sqrt{2t}\,\mathbf h_{N,\ell}$, yield
\begin{align*}
 \frac1{\sqrt{2t}}\E[\bg\cdot\mathbf h_N(\sigma_1)]
 &=\frac1{2t}\sum_\ell\E\partial_{\bg_\ell}^2\Phi(0,\bX_0;\bg)\\
 &=\int_0^1\E\left[|\mathbf h_N(\sigma_1)|^2
 -\left|\E[\mathbf h_N(\sigma_1)\mid\mathcal G_s]\right|^2\right]\dd s.
\end{align*}
The self-overlap terms cancel, and conditional independence and~\eqref{eq:orbit-realization} give
\begin{equation*}
 \partial_t\bar F_N^\nu(t,q)
 =\frac1N\int_0^1\E\left|\E[\mathbf h_N(\sigma_1)\mid\mathcal G_s]\right|^2\dd s
 =\int_0^1\E\xi(\overlap_s)\dd s.
\end{equation*}

For the path derivative, we first take $q$ smooth; the approximation argument of~\cite[Lemma 5.5, Steps 1--5]{chen2024simultaneous}, together with Lemma~\ref{lem:regularity}, gives the general case. Fix $\kappa\in C_c^\infty((0,1);S^D)$. Since $\dot q\succeq\varepsilon I_D$, the paths $q+h\kappa$ are admissible for small $|h|$. Summing the diagonal $D$-by-$D$ blocks in Lemma~\ref{lem:Parisi-Hessian}(3) and using conditional independence give
\begin{equation*}
 \sum_{i=1}^N\E\nabla^2_{\bx_{\cdot,i}}\Phi(s,\bX_s;\bg)+Ns\E\overlap_s
 =\E[\sigma_1\sigma_1^\intercal]-N\int_s^1\E\overlap_r\dd r.
\end{equation*}
Differentiating~\eqref{eq:Parisi-PDE} in $h$, applying It\^o's formula along $\bX$, and using Lemma~\ref{lem:Parisi-Hessian}(2) therefore yield
\begin{equation}\label{eq:path-variation-Phi}
 \begin{aligned}
 \left.\partial_h\bar F_N^\nu(t,q+h\kappa)\right|_{h=0}
 &=-\frac1N\int_0^1\dot\kappa(s)\cdot
 \left(\E[\sigma_1\sigma_1^\intercal]-N\int_s^1\E\overlap_r\dd r\right)\dd s\\
 &=\int_0^1\kappa(s)\cdot\E\overlap_s\dd s.
 \end{aligned}
\end{equation}
The last equality follows by integration by parts, since $\kappa(0)=\kappa(1)=0$. Density of these directions in $L^2([0,1];S^D)$ and Lemma~\ref{lem:regularity} identify the path derivative almost everywhere.
\end{proof}

\subsection{Concentration of the overlap and proof of Proposition~\ref{prop:residual}}

We prove the concentration of the conditional overlap introduced in Lemma~\ref{lem:diffusion-derivatives}.

\begin{lemma}\label{lem:overlap-concentration}
There is a constant $C_D<\infty$, depending only on $D$, such that
\begin{equation}\label{eq:total-concentration}
 \int_0^1\E|\overlap_s-\E\overlap_s|^2\dd s
 \le C_D(N\varepsilon)^{-1/3}.
\end{equation}
\end{lemma}
\begin{proof}
Since $|\overlap_s|\le1$, the left-hand side of~\eqref{eq:total-concentration} is at most $1$. It therefore suffices to consider $N\varepsilon>1$. Fix $\delta\in(0,1)$.

\smallskip
\noindent\emph{Step 1: Thermal fluctuations.}
For a deterministic $v\in\R^{DN}$, Lemma~\ref{lem:Parisi-Hessian}(3) gives
\begin{equation}\label{eq:linear-thermal-variance}
 \Var(v\cdot\operatorname{vec}(\sigma_1)\mid\mathcal G_s)
 =v^\intercal \mathrm{Cov}(\operatorname{vec}(\sigma_1)\mid\mathcal G_s)v\le\frac1s v^\intercal \nabla_{\bx}^2\Phi(s,\bX_s;\bg)v
 \le\frac{|v|^2}{s}\bigl|\nabla_{\bx}^2\Phi(s,\bX_s;\bg)\bigr|.
\end{equation}
With one replica fixed at $\tau$, each entry of $\overlap_s$ is a linear form in the other replica. The sum of the squared coefficient norms is
\begin{equation}\label{eq:coefficient-sum}
 \frac1{N^2}\sum_{\alpha,\beta=1}^D\sum_{i=1}^N\tau_{\alpha i}^2
 =\frac D{N^2}|\tau|^2\le\frac DN.
\end{equation}
Conditional variance tensorization over the two replicas, followed by summation over the entries, therefore gives
\begin{equation}\label{eq:thermal-pointwise}
 \E[|\overlap_s-\E[\overlap_s\mid\mathcal G_s]|^2\mid\mathcal G_s]
 \le\frac{2D}{Ns}\bigl|\nabla_{\bx}^2\Phi(s,\bX_s;\bg)\bigr|.
\end{equation}
Taking expectations, integrating, and applying Cauchy--Schwarz and Lemma~\ref{lem:Parisi-Hessian}(4), we obtain
\begin{align}\label{eq:thermal-concentration}
 \int_\delta^1\E|\overlap_s-\E[\overlap_s\mid\mathcal G_s]|^2\dd s
 &\le\frac{2D}{N}
 \left(\int_\delta^1s^{-2}\dd s\right)^{1/2}
 \left(\int_\delta^1\E\bigl|\nabla_{\bx}^2\Phi(s,\bX_s;\bg)\bigr|^2\dd s\right)^{1/2}
 \le\frac{C_D}{\sqrt{N\varepsilon\delta}}.
\end{align}

\smallskip
\noindent\emph{Step 2: Concentration of the conditional mean.}

Fix $a\in S^D$ with $|a|=1$ and let $T_a:=I_N\otimes a$, a self-adjoint contraction in column-stacked coordinates.

For $s\in(\delta,1]$, define
\begin{equation}\label{eq:U-definition}
 U_s:=\frac1{N\delta}\int_{s-\delta}^s
 a\cdot\left(\sigma_1 \E[\sigma_1\mid\mathcal G_r]^\intercal-\frac12\E[\sigma_1\mid\mathcal G_r]\E[\sigma_1\mid\mathcal G_r]^\intercal\right)\dd r.
\end{equation}
Since $\E[\sigma_1\mid\mathcal G_r]$ is $\mathcal G_s$-measurable for $r\le s$, the tower property and symmetry of $a$ give
\begin{equation}\label{eq:A-from-U}
 a\cdot \E[\overlap_s\mid\mathcal G_s]-2\E[U_s\mid\mathcal G_s]
 =\frac1{N\delta}\int_{s-\delta}^s
 a\cdot\bigl((\E[\sigma_1\mid\mathcal G_s]-\E[\sigma_1\mid\mathcal G_r])(\E[\sigma_1\mid\mathcal G_s]-\E[\sigma_1\mid\mathcal G_r])^\intercal\bigr)\dd r.
\end{equation}
The martingale identity from Lemma~\ref{lem:Parisi-Hessian}(2), 
\[\E|\E[\sigma_1\mid\mathcal G_s]-\E[\sigma_1\mid\mathcal G_r]|^2=\E|\E[\sigma_1\mid\mathcal G_s]|^2-\E|\E[\sigma_1\mid\mathcal G_r]|^2,\]  
the bounds $0\le\E|\E[\sigma_1\mid\mathcal G_r]|^2\le N$ and $|\E[\sigma_1\mid\mathcal G_s]-\E[\sigma_1\mid\mathcal G_r]|^2\le4N$, and time translation imply
\begin{align}\label{eq:integrated-martingale-increments}
 \int_{\delta}^1\E|a\cdot \E[\overlap_s\mid\mathcal G_s]-2\E[U_s\mid\mathcal G_s]|^2\dd s
 &\le\frac4{N\delta}\int_0^\delta\int_{\delta}^1
 \bigl(\E|\E[\sigma_1\mid\mathcal G_s]|^2-\E|\E[\sigma_1\mid\mathcal G_{s-v}]|^2\bigr)\dd s\dd v\nonumber\\
 &\le\frac4\delta\int_0^\delta v\dd v=2\delta.
\end{align}

Equation~\eqref{eq:orbit-equivariance} and invariance of $\lambda_N(\nu)$ give the terminal equivariance of $\Phi$. Uniqueness for the PDE then gives
\begin{equation}\label{eq:Phi-equivariance}
 \Phi(r,\bx^\pi;Q_\pi \bg)=\Phi(r,\bx;\bg),\qquad \pi\in\mathfrak S_N.
\end{equation}

Consequently the joint law is invariant under
\begin{equation}\label{e.(g,z,B,sigma)mapsto}
 (\bg,\bz_0,B,\sigma_1)\longmapsto(Q_\pi \bg,\bz_0^\pi,B^\pi,\sigma_1^\pi).
\end{equation}
The Gaussian inputs keep their laws, the SDE solution is column-permuted, and the terminal Gibbs law is carried to its corresponding permutation. Thus $\E[\sigma_1\mid\mathcal G_r]$ is also column-permuted and $U_s$ is unchanged. Since $\mathfrak S_N$ acts transitively on $\Sigma_N(\nu)$, $\E[U_s\mid\sigma_1]$ is constant, so
\begin{equation}\label{eq:conditional-centering}
 \Var(U_s)=\E\Var(U_s\mid\sigma_1).
\end{equation}
Centering~\eqref{eq:A-from-U}, applying Jensen's inequality, and using $|\E[\overlap_s\mid\mathcal G_s]|\le1$ on $[0,\delta]$ yield
\begin{equation}\label{eq:mean-average-reduction}
 \int_0^1\Var(a\cdot \E[\overlap_s\mid\mathcal G_s])\dd s
 \le C\delta+C\int_{\delta}^1\E\Var(U_s\mid\sigma_1)\dd s.
\end{equation}

We next bound the conditional variance in~\eqref{eq:mean-average-reduction}. 
Let $\Gamma$ be the product law of the independent inputs $(\bg,\bz_0,B)$, and put
\begin{equation*}
 \bY_r:=\sqrt{2q(0)}\bz_0+\int_0^r\sqrt{2\dot q(u)}\,\mathrm dB_u.
\end{equation*}
It\^o's formula for $r\Phi(r,\bY_r;\bg)$ identifies
\begin{equation}\label{eq:Girsanov-density}
 \exp\left(\Phi(1,\bY_1;\bg)-\int_0^1\Phi(r,\bY_r;\bg)\dd r\right)
\end{equation}
as the stochastic exponential with integrand $r\sqrt{2\dot q(r)}\nabla_{\bx}\Phi(r,\bY_r;\bg)$. The bounded gradient satisfies Novikov's condition~\cite[Chapter 3, Corollary 5.13]{karatzas1991brownian}, and Girsanov's theorem identifies $(\bg,\bz_0,\bY)$ under this change of measure with $(\bg,\bz_0,\bX)$ in~\eqref{eq:diffusion}. Multiplying by the terminal Gibbs mass at $\sigma\in\Sigma_N(\nu)$ cancels $\exp(\Phi(1,\bY_1;\bg))$. Hence the conditional law $\P_\sigma$ of the reference inputs given $\sigma_1=\sigma$ has density
\begin{equation}\label{eq:conditional-density}
 \frac{\mathrm d\P_\sigma}{\mathrm d\Gamma}
 \propto\exp\left(
 \sqrt{2t}\,\mathbf h_N(\sigma)\cdot \bg+\bY_1\cdot\sigma
 -\int_0^1\Phi(r,\bY_r;\bg)\dd r\right).
\end{equation}
The normalizing constant is positive and finite; terms depending only on $\sigma$ are absorbed into it. In these reference coordinates, evaluate $U_s$ by representing $\E[\sigma_1\mid\mathcal G_r]$ as $\nabla_{\bx}\Phi(r,\bY_r;\bg)$ using Lemma~\ref{lem:Parisi-Hessian}(2); the change of measure identifies this with the original conditional mean.

Choose an orthonormal basis of $L^2([0,1];\R^{D\times N})$ and write $(\eta_k)_{k\in\N}$ for the corresponding independent standard Gaussian coordinates of the Brownian noise. Put $\mathcal F_n:=\sigma(\eta_{n+1},\eta_{n+2},\ldots)$.
Let $\zeta$ collect $\bg,\bz_0,\eta_1,\ldots,\eta_n$, and condition on $\mathcal F_n$. Under $\Gamma$, $\zeta$ is standard Gaussian and $(\bY_r,\bg)$ is affine in $\zeta$. Under $\P_\sigma$, its conditional Lebesgue density is proportional to $\exp(-V_{\sigma,n})$, where
\begin{equation*}
 V_{\sigma,n}(\zeta)=\frac12|\zeta|^2-\sqrt{2t}\,\mathbf h_N(\sigma)\cdot \bg-\bY_1\cdot\sigma
 +\int_0^1\Phi(r,\bY_r;\bg)\dd r.
\end{equation*}
The middle terms are affine, and Lemma~\ref{lem:Parisi-Hessian}(1) gives
\begin{equation}\label{eq:conditional-potential-Hessian}
 \nabla_\zeta^2V_{\sigma,n}
 =I+\int_0^1\nabla_\zeta^2[\Phi(r,\bY_r;\bg)]\dd r\succeq I.
\end{equation}

For fixed $\sigma$, the derivative of $a\cdot(\sigma\E[\sigma_1\mid\mathcal G_r]^\intercal-\frac12\E[\sigma_1\mid\mathcal G_r]\E[\sigma_1\mid\mathcal G_r]^\intercal)$ with respect to its conditional-mean argument is $a(\sigma-\E[\sigma_1\mid\mathcal G_r])$. Cauchy--Schwarz for the joint Hessian of $\Phi$ and then in time therefore gives, for every $\zeta$-direction $w$,
\begin{equation}\label{eq:gradient-bar-U}
 \begin{gathered}
 |w\cdot\nabla_\zeta U_s|^2\le b_s\,w^\intercal (\nabla_\zeta^2V_{\sigma,n})w,\\
 b_s:=\frac1{N^2\delta^2}\int_{s-\delta}^s
 \operatorname{vec}(a(\sigma-\E[\sigma_1\mid\mathcal G_r]))^\intercal \nabla_{\bx}^2\Phi(r,\bY_r;\bg)\operatorname{vec}(a(\sigma-\E[\sigma_1\mid\mathcal G_r]))\dd r.
 \end{gathered}
\end{equation}
Indeed, the first quadratic factor is $w^\intercal\nabla_\zeta^2[\Phi(r,\bY_r;\bg)]w$, whose time integral is bounded by $w^\intercal (\nabla_\zeta^2V_{\sigma,n})w$. Taking $w=(\nabla_\zeta^2V_{\sigma,n})^{-1}\nabla_\zeta U_s$ and applying the Brascamp--Lieb variance inequality~\cite[Section 4]{brascampLieb1976} give
\begin{equation}\label{eq:BL-finite}
 \Var_{\P_\sigma}(U_s\mid\mathcal F_n)
 \le\E_{\P_\sigma}[\nabla_\zeta U_s^\intercal (\nabla_\zeta^2V_{\sigma,n})^{-1}\nabla_\zeta U_s\mid\mathcal F_n]
 \le\E_{\P_\sigma}[b_s\mid\mathcal F_n].
\end{equation}
Taking $\P_\sigma$-expectations first and using total variance yields
\begin{equation*}
 \Var_{\P_\sigma}(U_s)
 -\Var_{\P_\sigma}\bigl(\E_{\P_\sigma}[U_s\mid\mathcal F_n]\bigr)
 \le\E_{\P_\sigma}b_s.
\end{equation*}
The tail sigma-algebra is trivial under $\Gamma$ and hence under the equivalent measure $\P_\sigma$. 
Reverse martingale convergence
(see Theorem 4.7.3~in~\cite{durrett2019probability}) gives convergence in $L^1(\P_\sigma)$. Since $U_s$ and its conditional expectations are uniformly bounded, squared differences are bounded by a constant times absolute differences, so convergence also holds in $L^2(\P_\sigma)$ and the second variance vanishes. The right-hand side is independent of $n$ and integrable, since the spatial Hessian is bounded. Returning to the original diffusion $\bX$, with $\bY$ replaced by $\bX$ in $b_s$, gives
\begin{equation}\label{eq:BL-conditional}
 \Var(U_s\mid\sigma_1=\sigma)\le\E[b_s\mid\sigma_1=\sigma].
\end{equation}
Conditioning the quadratic integrand on $\mathcal G_r$ gives
\begin{equation}\label{eq:quadratic-covariance-bound}
 \begin{aligned}
 &\E[\operatorname{vec}(a(\sigma_1-\E[\sigma_1\mid\mathcal G_r]))^\intercal \nabla_{\bx}^2\Phi(r,\bX_r;\bg)\operatorname{vec}(a(\sigma_1-\E[\sigma_1\mid\mathcal G_r]))\mid\mathcal G_r]\\
 &=\tr(\nabla_{\bx}^2\Phi(r,\bX_r;\bg)T_a\mathrm{Cov}(\operatorname{vec}(\sigma_1)\mid\mathcal G_r)T_a)\le\frac1r\bigl|\nabla_{\bx}^2\Phi(r,\bX_r;\bg)\bigr|^2.
 \end{aligned}
\end{equation}
Here Lemma~\ref{lem:Parisi-Hessian}(3) gives
\begin{equation*}
 \mathrm{Cov}(\operatorname{vec}(\sigma_1)\mid\mathcal G_r)
 \preceq r^{-1}\nabla_{\bx}^2\Phi(r,\bX_r;\bg).
\end{equation*}
Since $T_a$ is a contraction,
\begin{align*}
 \tr(\nabla_{\bx}^2\Phi(r,\bX_r;\bg)T_a\nabla_{\bx}^2\Phi(r,\bX_r;\bg)T_a)
 &\le|\nabla_{\bx}^2\Phi(r,\bX_r;\bg)|\,|T_a\nabla_{\bx}^2\Phi(r,\bX_r;\bg)T_a|\\
 &\le\bigl|\nabla_{\bx}^2\Phi(r,\bX_r;\bg)\bigr|^2.
\end{align*}
Averaging~\eqref{eq:BL-conditional} over $\sigma$, applying Fubini and using Lemma~\ref{lem:Parisi-Hessian}(4) yield
\begin{align}\label{eq:U-variance}
 \int_{\delta}^1\E\Var(U_s\mid\sigma_1)\dd s
 &\le\frac1{N^2\delta^2}\int_{\delta}^1\int_{s-\delta}^s\frac1r\E\bigl|\nabla_{\bx}^2\Phi(r,\bX_r;\bg)\bigr|^2\dd r\dd s
 \nonumber \\
 &\le\frac1{N^2\delta^2}\int_0^1\frac{\min\{r,\delta\}}r\E\bigl|\nabla_{\bx}^2\Phi(r,\bX_r;\bg)\bigr|^2\dd r
 \nonumber \\
 &\le\frac1{2N\varepsilon\delta^2}.
\end{align}
Here we used $|[\delta,1]\cap[r,r+\delta]|\le\min\{r,\delta\}$ for $r>0$. Substituting into~\eqref{eq:mean-average-reduction} and summing over an orthonormal basis of $S^D$ give
\begin{equation}\label{eq:mean-concentration}
 \int_0^1\E|\E[\overlap_s\mid\mathcal G_s]-\E\overlap_s|^2\dd s
 \le C_D\left(\delta+\frac1{N\varepsilon\delta^2}\right).
\end{equation}

\smallskip
\noindent\emph{Step 3: Combination of the estimates.}
Conditional variance decomposition gives
\begin{equation}\label{eq:overlap-orthogonal}
 \E|\overlap_s-\E\overlap_s|^2=\E|\overlap_s-\E[\overlap_s\mid\mathcal G_s]|^2+\E|\E[\overlap_s\mid\mathcal G_s]-\E\overlap_s|^2.
\end{equation}
Since $|\overlap_s|,|\E[\overlap_s\mid\mathcal G_s]|\le1$, 
\eqref{eq:thermal-concentration} in Step 1 and \eqref{eq:mean-concentration} 
in Step 2
imply
\begin{equation*}
 \int_0^1\E|\overlap_s-\E\overlap_s|^2\dd s
 \le C_D\left(\delta+\frac1{N\varepsilon\delta^2}+\frac1{\sqrt{N\varepsilon\delta}}\right)
 \le \frac{3 C_D}{2}\left(\delta+\frac1{N\varepsilon\delta^2}\right).
\end{equation*}
Taking $\delta=(N\varepsilon)^{-1/3}$ proves~\eqref{eq:total-concentration}.
\end{proof}

\begin{proof}[Proof of Proposition~\ref{prop:residual}]
By Lemma~\ref{lem:diffusion-derivatives} and Taylor's theorem,
\begin{equation}\label{eq:Taylor-residuals}
 \left|\partial_t\bar F_N^\nu(t,q)-\int_0^1\xi(\partial_q\bar F_N^\nu(t,q,s))\dd s\right|
 \le\frac{M_2}{2}\int_0^1\E|\overlap_s-\E\overlap_s|^2\dd s.
\end{equation}
Lemma~\ref{lem:overlap-concentration} now gives~\eqref{eq:residual}.
\end{proof}
\subsection{A finite-union estimate}

We now allow arbitrary $N\in\N$, $t\ge0$, and $q\in\cQ_\infty$.

For a nonzero finite reference measure $\rho$ on $(\supp\mu)^N$, define $\Lambda_{\bg}$ with $\rho$ in place of $\lambda_N(\nu)$. For smooth increasing $q$, the representation~\eqref{e.F_N=-1/NEPhi} holds with $\bar F_{N,\rho}$ in place of $\bar F_N^\nu$, and we have
\begin{equation}\label{eq:diffusion-gradient-bounds}
 |\nabla_{\bx_{\cdot,i}}\Phi|\le1,
 \qquad |\nabla_{\bg}\Phi|^2\le2t\sup_{\sigma\in (\supp\mu)^N}|\mathbf h_N(\sigma)|^2\le2tNM_0.
\end{equation}
Indeed, a directional derivative in $\bx$ or $\bg$ solves the linear backward equation obtained by differentiating~\eqref{eq:Parisi-PDE}. Its terminal value is the Gibbs average of the corresponding linear observable. The maximum principle preserves these terminal bounds. For degenerate $q$, apply the argument to $q(s)+\gamma sI_D$ and let $\gamma\downarrow0$ using the cascade representation and Lemma~\ref{lem:regularity}. 

The following estimate compares the full free energy with the minimum over empirical measures and is adapted from~\cite[Proof of Lemma 3.4, equations (3.17)--(3.18)]{ho2026convergence}.

\begin{lemma}\label{lem:type-minimum}
Suppose that $\mu$ has finite support. For $T,Q\in\R_+$, put $C:=TM_0+(1+\sqrt D)Q$. For every $N\in\N$, $t\in[0,T]$, $q\in\cQ_\infty$ with $|q|_{L^\infty}\le Q$, and $s_0\in(0,1/2]$, we have
\begin{equation}\label{eq:type-minimum}
 0\le\min_{\nu\in\Delta_{m,N}}\bar F_{N,\mu^{\otimes N},\Sigma_N(\nu)}(t,q)-\bar F_{N,\mu^{\otimes N}}(t,q)
 \le Cs_0+\frac{m\log(N+1)}{Ns_0}.
\end{equation}
\end{lemma}
\begin{proof}
Restriction decreases the partition function, which proves the first inequality. By smooth approximation and Lemma~\ref{lem:regularity}, it suffices to prove the second inequality for smooth $q$.

In this proof, $\Phi$ and $\Phi_\nu$ denote the solutions of~\eqref{eq:Parisi-PDE} with reference measures $\mu^{\otimes N}$ and $\1_{\Sigma_N(\nu)}\mu^{\otimes N}$, respectively, using the Gaussian process fixed in~\eqref{eq:orbit-realization}. The restrictions retain their original masses as in~\eqref{eq:restricted-reference-normalization}. Since the sets $\Sigma_N(\nu)$ partition $(\supp\mu)^N$,
\begin{equation}\label{eq:type-terminal-softmax}
 \Phi(1,\bx;\bg)=\log\sum_{\nu\in\Delta_{m,N}}\exp\bigl(\Phi_\nu(1,\bx;\bg)\bigr).
\end{equation}
The representation~\eqref{e.F_N=-1/NEPhi} gives
\begin{align}\label{eq:type-Parisi-representation}
 \bar F_{N,\mu^{\otimes N}}(t,q)
 &=-\frac1N\E\Phi(0,\sqrt{2q(0)}\bz_0;\bg),\nonumber\\
 \bar F_{N,\mu^{\otimes N},\Sigma_N(\nu)}(t,q)
 &=-\frac1N\E\Phi_\nu(0,\sqrt{2q(0)}\bz_0;\bg),
\end{align}
where $\bz_0$ is standard Gaussian in $\R^{D\times N}$ and independent of $\bg$.

\smallskip
\noindent\emph{Comparison on $[s_0,1]$.}
For $s\in[s_0,1]$, let
\begin{equation*}
 \widehat\Phi(s,\bx;\bg):=
 \frac1s\log\sum_{\nu\in\Delta_{m,N}}e^{s\Phi_\nu(s,\bx;\bg)},
 \qquad
 p_\nu:=\frac{e^{s\Phi_\nu}}{\sum_{\nu'}e^{s\Phi_{\nu'}}}.
\end{equation*}
We claim that for all $(s,\bx)\in[s_0,1]\times \R^{D\times N}$
\begin{equation}\label{eq:type-softmax-comparison}
 \Phi(s,\bx;\bg)\le\widehat\Phi(s,\bx;\bg).
\end{equation}
This follows from the following standard maximum principle argument.

We have $\nabla_{\bx}\widehat\Phi=\sum_\nu p_\nu\nabla_{\bx}\Phi_\nu$, and differentiation gives
\begin{align*}
 \partial_s\widehat\Phi
 &=\sum_\nu p_\nu\partial_s\Phi_\nu+\frac1{s^2}\sum_\nu p_\nu\log p_\nu,\\
 \sum_{i=1}^N\dot q(s)\cdot\nabla^2_{\bx_{\cdot,i}}\widehat\Phi
 &=\sum_\nu p_\nu\sum_{i=1}^N\dot q(s)\cdot\nabla^2_{\bx_{\cdot,i}}\Phi_\nu
 +s\dot q(s)\cdot\Bigl(\sum_\nu p_\nu\nabla_{\bx}\Phi_\nu(\nabla_{\bx}\Phi_\nu)^\intercal
 -\nabla_{\bx}\widehat\Phi(\nabla_{\bx}\widehat\Phi)^\intercal\Bigr).
\end{align*}
Substituting the equations for $\Phi_\nu$, the quadratic terms cancel:
\begin{equation}\label{eq:type-softmax-PDE}
 \partial_s\widehat\Phi
 +\sum_{i=1}^N\dot q(s)\cdot\nabla^2_{\bx_{\cdot,i}}\widehat\Phi
 +s\dot q(s)\cdot\nabla_{\bx}\widehat\Phi(\nabla_{\bx}\widehat\Phi)^\intercal
 =\frac1{s^2}\sum_\nu p_\nu\log p_\nu\le0.
\end{equation}
Let $\widetilde\Phi:=\widehat\Phi-\Phi$. By~\eqref{eq:type-terminal-softmax}, 
the terminal condition equals $\widetilde\Phi(1,\cdot\,;\bg)=0$. Subtracting the equation for $\Phi$ from~\eqref{eq:type-softmax-PDE} gives
\begin{equation*}
 \partial_s\widetilde\Phi
 +\sum_{i=1}^N\dot q(s)\cdot\nabla^2_{\bx_{\cdot,i}}\widetilde\Phi
 +\sum_{i=1}^N b_i\cdot\nabla_{\bx_{\cdot,i}}\widetilde\Phi\le0,
 \qquad
 b_i:=s\dot q(s)\bigl(\nabla_{\bx_{\cdot,i}}\widehat\Phi+\nabla_{\bx_{\cdot,i}}\Phi\bigr).
\end{equation*}
After the change of variables $\tau=1-s$, the parabolic maximum principle, together with the zero terminal value of $\widetilde\Phi$, gives $\widetilde\Phi\ge0$ on $[s_0,1]$.
This proves \eqref{eq:type-softmax-comparison}.

\smallskip
\noindent\emph{Control of $[0,s_0]$.}
Let $B$ be a standard Brownian motion in $\R^{D\times N}$ independent of $(\bg,\bz_0)$, and let
\begin{equation*}
 \bY_s:=\sqrt{2q(0)}\bz_0+\int_0^s\sqrt{2\dot q(r)}\,\mathrm dB_r.
\end{equation*}
For $\Psi\in\{\Phi\}\cup\{\Phi_\nu:\nu\in\Delta_{m,N}\}$, It\^o's formula, the Parisi PDE, and the column bound in~\eqref{eq:diffusion-gradient-bounds} give
\begin{align}\label{eq:type-short-time}
 0&\le\E\Psi(0,\bY_0;\bg)-\E\Psi(s_0,\bY_{s_0};\bg)\nonumber\\
 &=\E\int_0^{s_0} r\dot q(r)\cdot\nabla_{\bx}\Psi(r,\bY_r;\bg)\bigl(\nabla_{\bx}\Psi(r,\bY_r;\bg)\bigr)^\intercal\,\mathrm dr
 \le Ns_0\operatorname{tr}q(1)\le Ns_0\sqrt D\,Q.
\end{align}
Moreover, $\bY_{s_0}$ has the same law as $\sqrt{2q(s_0)}\bz$, with $\bz$ standard Gaussian and independent of $\bg$.

\smallskip
\noindent\emph{Gaussian concentration.}
By~\eqref{eq:diffusion-gradient-bounds} and the chain rule,
\begin{equation*}
 \left|\nabla_{(\bg,\bz)}\bigl[\Phi_\nu(s_0,\sqrt{2q(s_0)}\bz;\bg)\bigr]\right|^2
 \le2tNM_0+2N\|q(s_0)\|_{\mathrm{op}}
 \le2N(TM_0+Q).
\end{equation*}
Gaussian concentration therefore gives
\begin{align*}
 \log\E e^{s_0\Phi_\nu(s_0,\sqrt{2q(s_0)}\bz;\bg)}
 \le s_0\E\Phi_\nu(s_0,\sqrt{2q(s_0)}\bz;\bg)
 +Ns_0^2(TM_0+Q).
\end{align*}
By Jensen's inequality,
\begin{align}\label{eq:type-random-softmax}
 \E\widehat\Phi(s_0,\bY_{s_0};\bg)
 &\le\frac1{s_0}\log\sum_{\nu\in\Delta_{m,N}}\E e^{s_0\Phi_\nu(s_0,\sqrt{2q(s_0)}\bz;\bg)}\nonumber\\
 &\le\max_{\nu\in\Delta_{m,N}}\E\Phi_\nu(s_0,\sqrt{2q(s_0)}\bz;\bg)
 +\frac{\log|\Delta_{m,N}|}{s_0}+Ns_0(TM_0+Q).
\end{align}
Combining~\eqref{eq:type-softmax-comparison}--\eqref{eq:type-random-softmax}, and using $\E\Phi_\nu(s_0,\sqrt{2q(s_0)}\bz;\bg)\le\E\Phi_\nu(0,\bY_0;\bg)$ from~\eqref{eq:type-short-time}, we obtain
\begin{equation}\label{eq:type-Parisi-max}
 \E\Phi(0,\bY_0;\bg)
 \le\max_{\nu\in\Delta_{m,N}}\E\Phi_\nu(0,\bY_0;\bg)
 +\frac{\log|\Delta_{m,N}|}{s_0}
 +Ns_0\bigl(TM_0+(1+\sqrt D)Q\bigr).
\end{equation}
Together with~\eqref{eq:type-Parisi-representation} and $|\Delta_{m,N}|\le(N+1)^m$, this proves the second inequality in~\eqref{eq:type-minimum}.
\end{proof}

\section{Initial conditions for finite spin sets}\label{sec:finite-initial}

In this section, assume $\supp\mu=\{a_1,\ldots,a_m\}$ and retain the notation $\mu_j:=\mu(\{a_j\})>0$ for $j\in\{1,\ldots,m\}$. We equip $\cP(\supp\mu)$ with the $W_1$ topology. Since $\supp\mu$ is finite, this is equivalent to the Euclidean topology under the identification $\nu\mapsto(\nu_1,\ldots,\nu_m)$, where $\nu_j:=\nu(\{a_j\})$. The path $q$ belongs to $\cQ_1$; the Gibbs computations below are first made for $q\in\cQ_\infty$ and then extended by approximation.

For $h\in\R^m$, interpret $h$ as the function on $\supp\mu$ satisfying $h(a_j)=h_j$ for every $j\in\{1,\ldots,m\}$, and write $h\cdot\nu:=\sum_{j=1}^m h_j\nu_j$ for $\nu\in\cP(\supp\mu)$. By~\eqref{eq:initial-definition}, we have $\psi_{\mu,h}(q)=\bar F_{1,e^h\mu}(0,q)$. The same identity holds at every system size:
\begin{equation}\label{e.tensorize}
 \bar F_{N,(e^h\mu)^{\otimes N}}(0,q)=\psi_{\mu,h}(q),
 \qquad N\in\N;
\end{equation}
this is the tensorization property in~\cite[Proposition 3.2]{chen2025free}.

\begin{proposition}\label{prop:finite-initial}
Assume $\supp\mu=\{a_1,\ldots,a_m\}$ for some $m\in\N$.
For every $q\in\cQ_1$, there is a continuous function $g_q:\cP(\supp\mu)\to\R$ such that
\begin{gather}\label{eq:finite-initial-identification}
 g_q(\nu)=\sup_{h\in\R^m}\{\psi_{\mu,h}(q)+h\cdot\nu\}
 =\Ent(\nu\mid\mu)+\varphi^\nu(q),
 \qquad \nu\in\cP(\supp\mu),
\\
\label{eq:finite-initial-convergence}
 \lim_{N\to\infty}\sup_{\nu\in\Delta_{m,N}}
 |\bar F_{N,\mu^{\otimes N},\Sigma_N(\nu)}(0,q)-g_q(\nu)|=0.
\end{gather}
For every $Q\in\R_+$, the convergence in~\eqref{eq:finite-initial-convergence} is uniform over $q\in\cQ_\infty$ satisfying $|q|_{L^\infty}\le Q$. Moreover, $\varphi^\nu$ is $1$-Lipschitz in $L^1$, increasing along $\cQ_2^*$, and satisfies $0\le\varphi^\nu(q)\le|q|_{L^1}$.
\end{proposition}

\begin{proof}
We first fix $q\in\cQ_\infty$.

\smallskip
\noindent\emph{Step 1: Compactness on $\Delta_{m,N}$ and the dual identity.}
Define $G_{N,q}:\Delta_{m,N}\to\R$ by
\begin{equation*}
 G_{N,q}(\nu):=\bar F_{N,\mu^{\otimes N},\Sigma_N(\nu)}(0,q).
\end{equation*}
Combining~\eqref{eq:normalization}, Lemma~\ref{lem:type-continuity}, and~\eqref{eq:stirling}, we obtain, for $\nu,\nu'\in\Delta_{m,N}$,
\begin{align}\label{eq:finite-equicontinuity}
 |G_{N,q}(\nu)-G_{N,q}(\nu')|
 &\le 2|q|_{L^\infty}W_1(\nu,\nu')
 +|\Ent(\nu\mid\mu)-\Ent(\nu'\mid\mu)|
 +\frac{m\log(N+1)}N.
\end{align}
The map $\nu\mapsto\Ent(\nu\mid\mu)=\sum_{j=1}^m\nu_j\log(\nu_j/\mu_j)$ is continuous on the compact set $\cP(\supp\mu)$, with the convention $0\log0:=0$. Thus~\eqref{eq:finite-equicontinuity} gives a common asymptotic modulus of continuity. Moreover, Lemma~\ref{lem:regularity}, applied to the probability measure $\lambda_N(\nu)$ at $t=0$, gives $|\bar F_N^\nu(0,q)|\le |q|_{L^1}$. Together with~\eqref{eq:normalization} and~\eqref{eq:stirling}, this shows that $(G_{N,q})_{N\in\N}$ is uniformly bounded on its varying domains.

The grids $\Delta_{m,N}$ become dense in $\cP(\supp\mu)$ for $W_1$: rounding the masses of any $\nu\in\cP(\supp\mu)$ to multiples of $N^{-1}$ gives $\nu_N\in\Delta_{m,N}$ with $\lim_{N\to\infty}W_1(\nu_N,\nu)=0$. Consequently, a finite-net and diagonal argument using~\eqref{eq:finite-equicontinuity} shows that every subsequence of $(G_{N,q})$ has a further subsequence, still indexed by $N$, and a continuous function $g_q:\cP(\supp\mu)\to\R$ such that
\begin{equation}\label{eq:finite-subsequence}
 \lim_{N\to\infty}\sup_{\nu\in\Delta_{m,N}}|G_{N,q}(\nu)-g_q(\nu)|=0.
\end{equation}

We next identify $g_q$. Fix $h\in\R^m$ and let
\begin{equation*}
 Z_h:=\sum_{j=1}^m\mu_j e^{h_j}
 \qquad\text{and}\qquad
 \mu_h:=Z_h^{-1}e^h\mu.
\end{equation*}
The measure $\mu_h$ is a probability measure with support $\supp\mu$. For every $\nu\in\Delta_{m,N}$, by~\eqref{eq:type-set}, every $\sigma\in\Sigma_N(\nu)$ has exactly $N\nu_j$ coordinates equal to $a_j$. Hence the density of $\mu_h^{\otimes N}$ relative to $\mu^{\otimes N}$ is constant on $\Sigma_N(\nu)$ and equals $Z_h^{-N}e^{Nh\cdot\nu}$. Therefore,
\begin{equation}\label{eq:tilted-restricted}
 \bar F_{N,\mu_h^{\otimes N},\Sigma_N(\nu)}(0,q)
 =G_{N,q}(\nu)+\log Z_h-h\cdot\nu.
\end{equation}
Equation~\eqref{e.tensorize} also gives
\begin{equation}\label{e.F=psi+logZ}
 \bar F_{N,\mu_h^{\otimes N}}(0,q)=\psi_{\mu,h}(q)+\log Z_h.
\end{equation}
Applying Lemma~\ref{lem:type-minimum} to the reference measure $\mu_h$, with $t=0$ and $s_{0,N}:=\min\{\frac12,\sqrt{\log(N+1)/N}\}$, and using~\eqref{eq:tilted-restricted} and~\eqref{e.F=psi+logZ}, we obtain
\begin{equation*}
 \lim_{N\to\infty}
 \left|
 \min_{\nu\in\Delta_{m,N}}\{G_{N,q}(\nu)-h\cdot\nu\}
 -\psi_{\mu,h}(q)
 \right|=0.
\end{equation*}
Passing along the subsequence in~\eqref{eq:finite-subsequence}, using the density of $\Delta_{m,N}$ in $\cP(\supp\mu)$ and the continuity of $g_q$, gives
\begin{equation}\label{eq:finite-duality}
 \psi_{\mu,h}(q)=\min_{\nu\in\cP(\supp\mu)}\{g_q(\nu)-h\cdot\nu\},
 \qquad h\in\R^m.
\end{equation}

\smallskip
\noindent\emph{Step 2: Identification of $g_q$ by convex duality.}
For bounded $q$, differentiation with respect to $h$ gives
\begin{equation}\label{eq:psi-h-gradient}
 \partial_{h_j}\psi_{\mu,h}(q)
 =-\E\langle\1_{\{\sigma=a_j\}}\rangle,
 \qquad j\in\{1,\ldots,m\}.
\end{equation}
A second differentiation expresses the Hessian as the negative of an expected covariance matrix of the indicators $(\1_{\{\sigma=a_j\}})_{j=1}^m$. Since these indicators are bounded by one, the Hessian is uniformly bounded. Thus $h\mapsto\psi_{\mu,h}(q)$ is continuously differentiable.

Fix $h\in\R^m$, and let $\eta\in\cP(\supp\mu)$ be a minimizer in~\eqref{eq:finite-duality}; such a minimizer exists because $\cP(\supp\mu)$ is compact and $g_q$ is continuous. For every $v\in\R^m$ and $\varepsilon\in\R$, we have
\begin{equation*}
 \psi_{\mu,h+\varepsilon v}(q)
 \le g_q(\eta)-(h+\varepsilon v)\cdot\eta
 =\psi_{\mu,h}(q)-\varepsilon v\cdot\eta.
\end{equation*}
Dividing by $\varepsilon$ for $\varepsilon>0$ and $\varepsilon<0$ and then sending $\varepsilon\to0$ gives
\begin{equation}\label{eq:minimizer-gradient}
 \eta=-\nabla_h\psi_{\mu,h}(q).
\end{equation}

We now fix $\nu$ in the relative interior of $\cP(\supp\mu)$ and show that there is an $h$ for which $\nu$ is the minimizer in~\eqref{eq:finite-duality}. Equation~\eqref{eq:regularity} and the identity $\psi_{\mu,h}(0)=-\log\sum_{j=1}^m\mu_j e^{h_j}$ give
\begin{equation}\label{eq:psi-h-zero-comparison}
 \bigg|\psi_{\mu,h}(q)+\log\sum_{j=1}^m\mu_j e^{h_j}\bigg|
 \le |q|_{L^1}.
\end{equation}
Define $J_\nu(h):=\psi_{\mu,h}(q)+h\cdot\nu$. Translation invariance, $J_\nu(h+c\1)=J_\nu(h)$, allows us to normalize $\max_jh_j=0$. Put $\mu_*:=\min_j\mu_j>0$ and $\nu_*:=\min_j\nu_j>0$. On this slice,~\eqref{eq:psi-h-zero-comparison} gives
\begin{equation*}
 J_\nu(h)\le-\log\mu_*+|q|_{L^1}-\nu_*|h|_\infty.
\end{equation*}
Thus $J_\nu$ attains a maximum on the slice, and translation invariance makes the maximizer $h_*$ global on $\R^m$. This also applies when $m=1$, for which the slice is the single point $0$. Differentiability at the global maximum gives
\begin{equation*}
 \nu=-\nabla_h\psi_{\mu,h_*}(q).
\end{equation*}
This is precisely~\eqref{eq:minimizer-gradient} with $\eta=\nu$. Every minimizer in~\eqref{eq:finite-duality} at $h=h_*$ must therefore equal $\nu$. Consequently,
\begin{equation}\label{eq:g-legendre-interior}
 g_q(\nu)=\psi_{\mu,h_*}(q)+h_*\cdot\nu
 =\sup_{h\in\R^m}\{\psi_{\mu,h}(q)+h\cdot\nu\}
\end{equation}
for every $\nu$ in the relative interior of $\cP(\supp\mu)$.

Equation~\eqref{eq:g-legendre-interior} shows that $g_q$ is convex on the relative interior. Since $g_q$ is continuous on $\cP(\supp\mu)$, it is convex on all of $\cP(\supp\mu)$. Extend $g_q$ to $\R^m$ by setting it equal to $+\infty$ outside $\cP(\supp\mu)$. The resulting function is proper, lower semicontinuous, and convex. Equation~\eqref{eq:finite-duality} says that its convex conjugate is $h\mapsto-\psi_{\mu,h}(q)$. The finite-dimensional Fenchel--Moreau theorem therefore gives
\begin{equation}\label{eq:g-legendre}
 g_q(\nu)=\sup_{h\in\R^m}\{\psi_{\mu,h}(q)+h\cdot\nu\},
 \qquad \nu\in\cP(\supp\mu).
\end{equation}
The right-hand side of~\eqref{eq:g-legendre} is independent of the subsequence chosen in~\eqref{eq:finite-subsequence}. Hence every subsequence has the same subsequential limit, and the whole sequence satisfies~\eqref{eq:finite-initial-convergence}. Equation~\eqref{eq:g-legendre} proves the first equality in~\eqref{eq:finite-initial-identification}.

\smallskip
\noindent\emph{Step 3: Identification with $\varphi^\nu$ and uniformity in $q$.}
If $\nu_j>0$ for every $j$, change the tilt by
\begin{equation*}
 \widehat h_j:=h_j+\log(\mu_j/\nu_j).
\end{equation*}
Then $e^h\mu=e^{\widehat h}\nu$ and $h\cdot\nu=\widehat h\cdot\nu+\Ent(\nu\mid\mu)$, so~\eqref{eq:g-legendre} and~\eqref{eq:initial-definition} give $g_q(\nu)=\Ent(\nu\mid\mu)+\varphi^\nu(q)$ directly.

For a boundary measure $\nu$, restrict first to $\supp\nu$. For any $h\in\R^m$, removing the mass outside $\supp\nu$ increases the free energy, and the same tilt change on $\supp\nu$ gives
\begin{equation*}
 \psi_{\mu,h}(q)+h\cdot\nu\le\Ent(\nu\mid\mu)+\varphi^\nu(q).
\end{equation*}
Conversely, fix a tilt $k$ on $\supp\nu$ and let $h_j=k_j+\log(\nu_j/\mu_j)$ for $a_j\in\supp\nu$, and $h_j=-M$ for $a_j\notin\supp\nu$. As $M\to\infty$, the partition functions decrease to those with reference measure $e^k\nu$. Their logarithms lie between the integrable logarithms for $M=0$ and the restricted measure, so dominated convergence gives $\psi_{\mu,h}(q)\to\psi_{\nu,k}(q)$. Taking the supremum over $k$ proves
\begin{equation}\label{e.g_q=Ent+varphi}
 g_q(\nu)=\Ent(\nu\mid\mu)+\varphi^\nu(q),
 \qquad\nu\in\cP(\supp\mu).
\end{equation}
This proves the second equality in~\eqref{eq:finite-initial-identification}. The normalization identity~\eqref{eq:normalization} and the type-probability bound~\eqref{eq:stirling} transfer the uniform convergence in~\eqref{eq:finite-initial-convergence} to the normalized free energies. Since $g_q$ and entropy are continuous,
\begin{equation}\label{eq:normalized-constrained-limit}
 \lim_{N\to\infty}\bar F_N^{\nu_N}(0,q)=\varphi^\nu(q)
 \quad\text{whenever }\nu_N\in\Delta_{m,N},\quad\nu_N\to\nu\text{ in }W_1.
\end{equation}

By Lemma~\ref{lem:initial-regularity}, $\varphi^\nu$ is $1$-Lipschitz in $L^1$. Together with Lemma~\ref{lem:regularity}, this extends~\eqref{eq:finite-initial-identification} and~\eqref{eq:finite-initial-convergence} from bounded paths to every $q\in\cQ_1$ by $L^1$ approximation.

Finally, fix $Q\in\R_+$ and define
\begin{equation*}
 \mcl Q_{\infty, \leq Q}:=\{q\in\cQ_\infty:|q|_{L^\infty}\le Q\}.
\end{equation*}
The set $\mcl Q_{\infty, \leq Q}$ is relatively compact in $L^1$. Indeed, for every partition $0=s_0<\cdots<s_k<1$, positivity of the increments gives
\begin{equation*}
 \sum_{\ell=1}^k|q(s_\ell)-q(s_{\ell-1})|
 \le\sum_{\ell=1}^k\tr\bigl(q(s_\ell)-q(s_{\ell-1})\bigr)
 \le\sqrt D\,Q.
\end{equation*}
Helly's selection argument and dominated convergence therefore give $L^1$ compactness after taking right-continuous representatives. By~\eqref{eq:normalization} and Lemma~\ref{lem:regularity}, the map $q\mapsto\bar F_{N,\mu^{\otimes N},\Sigma_N(\nu)}(0,q)$ is $1$-Lipschitz in $L^1$, uniformly in $N$ and $\nu$. By~\eqref{e.g_q=Ent+varphi} and the $L^1$-Lipschitz bound for $\varphi^\nu$, the map $q\mapsto g_q(\nu)$ has the same bound, uniformly in $\nu$.

Fix $\varepsilon>0$ and choose $q^1,\ldots,q^R\in \mcl Q_{\infty, \leq Q}$ such that every $q\in \mcl Q_{\infty, \leq Q}$ satisfies $\min_{r\in\{1,\ldots,R\}}|q-q^r|_{L^1}<\varepsilon$. By~\eqref{eq:finite-initial-convergence}, for all sufficiently large $N$ we have
\begin{equation*}
 \max_{r\in\{1,\ldots,R\}}
 \sup_{\nu\in\Delta_{m,N}}
 |\bar F_{N,\mu^{\otimes N},\Sigma_N(\nu)}(0,q^r)-g_{q^r}(\nu)|
 <\varepsilon.
\end{equation*}
For $q\in \mcl Q_{\infty, \leq Q}$, choose $r$ with $|q-q^r|_{L^1}<\varepsilon$. The two Lipschitz estimates and the preceding display give
\begin{equation*}
 \sup_{\nu\in\Delta_{m,N}}
 |\bar F_{N,\mu^{\otimes N},\Sigma_N(\nu)}(0,q)-g_q(\nu)|
 <3\varepsilon.
\end{equation*}
Since $\varepsilon>0$ is arbitrary, we obtain
\begin{equation*}
 \lim_{N\to\infty}
 \sup_{\substack{q\in\cQ_\infty\\|q|_{L^\infty}\le Q}}
 \sup_{\nu\in\Delta_{m,N}}
 |\bar F_{N,\mu^{\otimes N},\Sigma_N(\nu)}(0,q)-g_q(\nu)|=0.
\end{equation*}
The proof is complete.
\end{proof}

\begin{lemma}\label{lem:finite-initial-W1}
Suppose that $\mu$ has finite support. For every $\nu,\nu'\in\cP(\supp\mu)$, the initial data satisfy
\begin{equation}\label{eq:finite-initial-W1}
 |\varphi^\nu(q)-\varphi^{\nu'}(q)|
 \le2|q|_{L^\infty}W_1(\nu,\nu'),\qquad q\in\cQ_\infty.
\end{equation}
\end{lemma}
\begin{proof}
Choose $\nu_N,\nu_N'\in\Delta_{m,N}$ converging to $\nu,\nu'$ in $W_1$. Lemma~\ref{lem:type-continuity} at $t=0$ gives
\begin{equation*}
 |\bar F_N^{\nu_N}(0,q)-\bar F_N^{\nu_N'}(0,q)|
 \le2|q|_{L^\infty}W_1(\nu_N,\nu_N').
\end{equation*}
Pass to the limit using~\eqref{eq:normalized-constrained-limit}.
\end{proof}

\section{Identification of the constrained viscosity solutions}\label{sec:constrained-HJ}

We now give a precise meaning to the Hamilton--Jacobi solution
$f^\nu$ appearing in Theorem~\ref{thm:main} and identify $f^\nu$ as
the limit of the normalized constrained free energies when $\mu$
has finite support. 

\subsection{Viscosity solutions and finite-dimensional approximations}
Let $E$ be a finite-dimensional Hilbert space, let $U\subset(0,\infty)\times E$ be open, and let $H:E\to\R$ be continuous.

\begin{definition}[Viscosity solutions]\label{def:viscosity}
A continuous function $u:U\to\R$ is a \textbf{viscosity subsolution} of $\partial_tu-H(D_xu)=0$ if, for every $\phi\in C^1(U)$ and every local maximum $(t,x)\in U$ of $u-\phi$, we have
\begin{equation*}
 \partial_t\phi(t,x)-H(D_x\phi(t,x))\le0.
\end{equation*}
A \textbf{viscosity supersolution} satisfies the reverse inequality at every local minimum of $u-\phi$. A \textbf{viscosity solution} is both a viscosity subsolution and a viscosity supersolution.
\end{definition}

We use Definition~\ref{def:viscosity} only for finite-dimensional equations; the solution on $\cQ_2$ is defined by finite-dimensional approximations below.

The spin-glass Hamilton--Jacobi equation that we wish to interpret is
\begin{equation}\label{eq:HJ}
 \partial_t f-\int_0^1\xi\bigl(\partial_qf(s)\bigr)\dd s=0
 \quad\text{on }\R_+\times\cQ_2.
\end{equation}
We define its solution through finite-dimensional approximations. For $K\in\N$, equip $E_K:=(S^D)^K$ with the inner product $\ip a b_K:=K^{-1}\sum_{k=1}^Ka_k\cdot b_k$, and write $|a|_{\ell^1_K}:=K^{-1}\sum_{k=1}^K|a_k|$ and $|a|_{\ell^\infty}:=\max_{k\in\{1,\dots,K\}}|a_k|$. Define the closed convex cone
\begin{equation*}
 \cQ^K:=\{a\in E_K:0\preceq a_1\preceq\cdots\preceq a_K\}.
\end{equation*}
As in~\eqref{eq:dual-cone}, define the dual cone $(\cQ^K)^*:=\{b\in E_K:\ip b a_K\ge0\text{ for every }a\in\cQ^K\}$. As for the notion of $\cQ_2^*$-increasing introduced after~\eqref{eq:dual-cone}, a function $u:\cQ^K\to\R$ is \textbf{$(\cQ^K)^*$-increasing} if $u(a)\ge u(b)$ whenever $a,b\in\cQ^K$ satisfy $a-b\in(\cQ^K)^*$.
For the uniform partition of $[0,1)$, define the lift $L_K:E_K\to L^2([0,1];S^D)$ and projection $P_K:L^2([0,1];S^D)\to E_K$ by
\begin{equation*}
 L_Ka:=\sum_{k=1}^Ka_k\1_{[(k-1)/K,k/K)}
 \quad\text{and}\quad
 (P_Kp)_k:=K\int_{(k-1)/K}^{k/K}p(s)\dd s.
\end{equation*}
These maps preserve the corresponding cones in the sense that $P_K(\cQ_2)=\cQ^K$ and $L_K(\cQ^K)\subseteq\cQ_2$; see~\cite[Lemmas 3.3 and 3.5]{chen2022hamilton}. Let $\nabla_K$ denote the gradient in the weighted inner product on $E_K$; thus it is $K$ times the ordinary Euclidean gradient. Define the finite-dimensional Hamiltonian $\xi_K:E_K\to\R$ by
\begin{equation}\label{eq:finite-H}
 \xi_K(p):=\frac1K\sum_{k=1}^K\xi(p_k).
\end{equation}

For every $\nu\in\cP(\supp\mu)$, Lemma~\ref{lem:initial-regularity} supplies the regularity and monotonicity of $\varphi^\nu$ required below.
The restriction of $\xi_K$ to $\cQ^K$ is locally Lipschitz and increasing along $(\cQ^K)^*$, by the covariance assumptions on $\xi$; see~\cite[proof of Theorem 4.7(1)]{chen2022hamilton}. Applying~\cite[Theorem 2.2(2)]{chen2022hamilton}, let $f^{\nu,K}:\R_+\times\cQ^K\to\R$ be the unique Lipschitz, $(\cQ^K)^*$-increasing viscosity solution, in the sense of Definition~\ref{def:viscosity} on $(0,\infty)\times\operatorname{int}\cQ^K$, of
\begin{equation}\label{eq:finite-HJ}
 \begin{cases}
 \partial_t f^{\nu,K}-\xi_K(\nabla_Kf^{\nu,K})=0 & \text{on }(0,\infty)\times\operatorname{int}\cQ^K,\\
 f^{\nu,K}(0,\cdot)=\varphi^\nu(L_K\cdot) & \text{on }\cQ^K.
 \end{cases}
\end{equation}
The monotonicity of $\xi_K$ along $(\cQ^K)^*$ makes a separate boundary condition unnecessary: existence and uniqueness follow from~\cite[Theorem 2.2]{chen2022hamilton}, and the continuous extension from the interior solves the equation on the closed cone in the sense of~\cite[Proposition 2.3]{chen2022hamilton}.

\begin{definition}[Solution of the spin-glass Hamilton--Jacobi equation]\label{def:spin-glass-solution}
For $\nu\in\cP(\supp\mu)$, define $f^\nu:\R_+\times\cQ_2\to\R$ by
\begin{equation}\label{eq:HJ-discrete-limit}
 f^\nu(t,q):=\lim_{K\to\infty}f^{\nu,K}(t,P_Kq),
 \qquad (t,q)\in\R_+\times\cQ_2.
\end{equation}
\end{definition}

The limit in~\eqref{eq:HJ-discrete-limit} exists locally uniformly. Indeed, the uniform-refinement estimate in~\cite[Proof of Proposition 3.10, (3.42)]{chen2022hamilton} allows one to compare the approximations on grids of sizes $K$ and $K'$ through their common refinement of size $\operatorname{lcm}(K,K')$. Its identification as the unique Lipschitz viscosity solution of~\eqref{eq:HJ} in the infinite-dimensional sense is proved in~\cite[Definition 4.3 and Theorem 4.7(1)]{chen2022hamilton}; see also~\cite[Remark 4.11 and Proposition 4.13]{chen2022hamilton}. We do not repeat the infinite-dimensional definition, since~\eqref{eq:finite-HJ} and~\eqref{eq:HJ-discrete-limit} suffice for our purposes.

By~\cite[Proposition 2.4 and Theorem 4.7(1)]{chen2022hamilton}, the finite-dimensional solutions preserve the $1$-Lipschitz constant in the weighted $\ell^1$ norm. Their time Lipschitz constants are at most $M_0$ (see~\eqref{e.M_i=}), because~\eqref{eq:finite-H} is bounded by $M_0$ on $\cQ^K\cap\{p:|p|_{\ell^\infty}\le1\}$. These bounds pass to~\eqref{eq:HJ-discrete-limit}; in particular, we have
\begin{equation}\label{eq:solution-regularity}
 |f^\nu(t,q)-f^\nu(t',q')|\le M_0|t-t'|+|q-q'|_{L^1},
 \qquad (t,q),(t',q')\in\R_+\times\cQ_2.
\end{equation}

\subsection{Finite-dimensional comparison}

For the rest of this section, assume that $\supp\mu=\{a_1,\ldots,a_m\}$ is finite. We retain the normalized constrained free energies $\bar F_N^\nu$ (see~\eqref{eq:normalized-constrained-notation}), with initial limits $\varphi^\nu$ from Proposition~\ref{prop:finite-initial}. We pass from the finite-size residual to viscosity comparison using the maps $L_K$ and $P_K$ defined above.

\begin{lemma}\label{lem:projection}
For every $p\in\cQ_\infty$ with $|p|_{L^\infty}\le1$, we have
\begin{equation*}
 |p-L_KP_Kp|_{L^1}\le\frac{\sqrt D}{K}
 \quad\text{and}\quad
 \left|\int_0^1\xi(p(s))\dd s-\xi_K(P_Kp)\right|\le\frac{M_1\sqrt D}{K}.
\end{equation*}
The Hamiltonian $\xi_K$ satisfies the Lipschitz bound
\begin{equation}\label{eq:finite-h-Lipschitz}
 |\xi_K(p)-\xi_K(p')|\le M_1|p-p'|_{\ell^1_K},
 \qquad p,p'\in\cQ^K,\quad|p|_{\ell^\infty}\vee|p'|_{\ell^\infty}\le1.
\end{equation}
\end{lemma}
\begin{proof}
Fix $p\in\cQ_\infty$ and let $I_k:=[(k-1)/K,k/K)$ for $k\in\{1,\ldots,K\}$. Since $p$ is increasing in the positive-semidefinite order, we have
\begin{equation*}
 |p(s)-p(r)|\le\tr\bigl(p(s)-p(r)\bigr),
 \qquad 0\le r\le s<1.
\end{equation*}
By Jensen's inequality and the definition of $P_K$, we obtain
\begin{align*}
 \int_{I_k}|p(s)-(P_Kp)_k|\dd s
 &\le K\int_{I_k}\int_{I_k}|p(s)-p(r)|\dd r\dd s\le\frac1K\biggl(\sup_{s\in I_k}\tr p(s)-\inf_{s\in I_k}\tr p(s)\biggr).
\end{align*}
The increments in the last line sum to at most $\sup_{s\in[0,1)}\tr p(s)$, because $\tr p$ is nonnegative and increasing. Consequently, for every $p\in\cQ_\infty$, we have
\begin{equation}\label{eq:projection-L1-bound}
 |p-L_KP_Kp|_{L^1}
 \le\frac1K\sup_{s\in[0,1)}\tr p(s)
 \le\frac{\sqrt D}{K}|p|_{L^\infty}.
\end{equation}
This proves the first assertion of the lemma when $|p|_{L^\infty}\le1$.

In this case, both $p(s)$ and $(L_KP_Kp)(s)$ belong to the unit ball of $S^D$. By~\eqref{e.M_i=}, the mean-value theorem gives
\begin{align*}
 \left|\int_0^1\xi(p(s))\dd s-\xi_K(P_Kp)\right|
 &\le M_1|p-L_KP_Kp|_{L^1}
 \stackref{\le}{\eqref{eq:projection-L1-bound}}\frac{M_1\sqrt D}{K}.
\end{align*}
Finally, for $p,p'\in\cQ^K$ satisfying $|p|_{\ell^\infty}\vee|p'|_{\ell^\infty}\le1$, the line segment between $p_k$ and $p_k'$ lies in the unit ball. The definition of $M_1$ gives
\begin{align*}
 |\xi_K(p)-\xi_K(p')|
 &\le\frac1K\sum_{k=1}^K|\xi(p_k)-\xi(p_k')|\le\frac{M_1}{K}\sum_{k=1}^K|p_k-p_k'|
 =M_1|p-p'|_{\ell^1_K}.
\end{align*}
This proves~\eqref{eq:finite-h-Lipschitz}.
\end{proof}

\begin{lemma}\label{lem:step-residual}
Proposition~\ref{prop:residual} also holds at $q=L_Ka+\varepsilon\ell$, where $a\in\cQ^K$ and $\ell:[0,1]\to S^D_+$ is given by $\ell(s):=sI_D$.
\end{lemma}
\begin{proof}
Fix $a\in\cQ^K$ and let $\widehat q_a:\R\to S^D_+$ be the extension of $L_Ka$ that is constant on $(-\infty,0)$ and $[1,\infty)$. Choose a nonnegative $\kappa\in C_c^\infty(\R)$ with $\int_\R\kappa=1$. For $\eta>0$, define $\kappa_\eta(r):=\eta^{-1}\kappa(r/\eta)$ and $q_\eta^0:=\kappa_\eta*\widehat q_a$. Since $\widehat q_a$ is positive semidefinite and increasing, the same holds for $q_\eta^0$. In addition, $q_\eta^0$ is smooth, and the 
    standard mollifying argument
    gives $\lim_{\eta\downarrow0}|q_\eta^0-L_Ka|_{L^1}=0$.
Consequently, $q_\eta:=q_\eta^0+\varepsilon\ell$ satisfies $\dot q_\eta(s)\succeq\varepsilon I_D$ for every $s\in[0,1]$ and $\lim_{\eta\downarrow0}|q_\eta-(L_Ka+\varepsilon\ell)|_{L^1}=0$.

We apply Proposition~\ref{prop:residual} to $q_\eta$ and let $\eta\downarrow0$, with $N,t,K,a,\varepsilon$ fixed. Equation~\eqref{eq:gradient-continuity} gives convergence of the path derivatives in $L^1$. The time derivative is continuous as well: the bounded replica representation in~\eqref{eq:replica-derivatives}, differentiated along a Gaussian interpolation, gives
\begin{equation*}
 |\partial_t\bar F_N^\nu(t,q)-\partial_t\bar F_N^\nu(t,q')|
 \le C NM_0|q-q'|_{L^1}.
\end{equation*}
Let $q_*:=L_Ka+\varepsilon\ell$. For fixed $N$, both derivatives therefore converge as $\eta\downarrow0$. Their path components have $L^\infty$ norm at most one by Lemma~\ref{lem:regularity}, so
\begin{equation*}
 \left|\int_0^1\xi(\partial_q\bar F_N^\nu(t,q_\eta,s))\dd s
 -\int_0^1\xi(\partial_q\bar F_N^\nu(t,q_*,s))\dd s\right|
 \le M_1|\partial_q\bar F_N^\nu(t,q_\eta)-\partial_q\bar F_N^\nu(t,q_*)|_{L^1}\longrightarrow0.
\end{equation*}
The residual bound at $q_\eta$ is independent of $\eta$ and passes to $q_*$.
\end{proof}

We now prove a \emph{quantitative} comparison principle, in the style of \cite[Proposition~2.5]{chen2022hamilton}. What is important here is that in the constraint on the $\ell^\infty$-norm of $b$ in \eqref{eq:finite-local-comparison} below the growth in time (given by $M_1$) is uniform in $K$.

\begin{lemma}[Comparison principle for viscosity solutions]\label{lem:finite-local-comparison}
Fix $K\in\N$ and $T,Q\in\R_+$. Let $u,v:\R_+\times\cQ^K\to\R$ be respectively a viscosity subsolution and a viscosity supersolution of $\partial_t w-\xi_K(\nabla_Kw)=0$ on $(0,\infty)\times\operatorname{int}\cQ^K$. Suppose that $u(t,\cdot)$ and $v(t,\cdot)$ are $1$-Lipschitz in $|\cdot|_{\ell^1_K}$ and $(\cQ^K)^*$-increasing for every $t\in\R_+$. Assume also that there exists $C\in\R_+$ such that
\begin{equation*}
 |u(t,a)-u(0,a)|\vee|v(t,a)-v(0,a)|\le Ct,
 \qquad(t,a)\in\R_+\times\cQ^K.
\end{equation*}
Then
\begin{equation}\label{eq:finite-local-comparison}
 \sup_{\substack{0\le t\le T,\ a\in\cQ^K\\|a|_{\ell^\infty}\le Q}}
 (u(t,a)-v(t,a))
 \le
 \sup_{\substack{b\in\cQ^K\\|b|_{\ell^\infty}\le Q+M_1T+1}}
 (u(0,b)-v(0,b)).
\end{equation}
\end{lemma}
The proof of Lemma~\ref{lem:finite-local-comparison} is a technical adaptation of the standard doubling-of-variables argument. The main difficulty is to establish comparison without imposing a boundary condition, which requires an auxiliary penalization to ensure that the contact points lie in the interior of the cone. We defer the details to Appendix~\ref{app:finite-local-comparison}.

Assume that $\mu$ has finite support, fix $T,Q\in\R_+$, and put $Q':=Q+M_1T+1$.

\begin{proposition}\label{prop:finite-comparison} 
We have
\begin{equation}\label{eq:initial-discrepancy}
 \lim_{N\to\infty}\sup_{\substack{\nu'\in\Delta_{m,N},\ p\in\cQ_\infty\\|p|_{L^\infty}\le Q'}}|\bar F_N^{\nu'}(0,p)-\varphi^{\nu'}(p)|=0.
\end{equation}
For every $N\in\N$, $\nu\in\Delta_{m,N}$, $t\in[0,T]$, $a\in\cQ^K$ with $|a|_{\ell^\infty}\le Q$, and $\varepsilon>0$, we have
\begin{equation}\label{eq:finite-comparison}
 \begin{aligned}
 |\bar F_N^\nu(t,L_Ka)-f^{\nu,K}(t,a)|
 &\le \sup_{\substack{\nu'\in\Delta_{m,N},\ p\in\cQ_\infty\\|p|_{L^\infty}\le Q'}}|\bar F_N^{\nu'}(0,p)-\varphi^{\nu'}(p)|\\
 &\quad+\varepsilon\sqrt D
 +t\left(C(N\varepsilon)^{-1/3}+\frac{M_1\sqrt D}{K}\right).
 \end{aligned}
\end{equation}
\end{proposition}
\begin{proof}
By~\eqref{eq:normalization},~\eqref{eq:finite-initial-identification}, and~\eqref{eq:stirling}, for every $\nu'\in\Delta_{m,N}$ and $p\in\cQ_\infty$ with $|p|_{L^\infty}\le Q'$, we have
\begin{equation*}
 |\bar F_N^{\nu'}(0,p)-\varphi^{\nu'}(p)|
 \le|\bar F_{N,\mu^{\otimes N},\Sigma_N(\nu')}(0,p)-g_p(\nu')|
 +\frac{m\log(N+1)}N.
\end{equation*}
Taking suprema and using the uniform initial convergence in Proposition~\ref{prop:finite-initial} prove~\eqref{eq:initial-discrepancy}.

\smallskip
\noindent\emph{Step 1: Estimate for $\partial_tV_N-\xi_K(\nabla_KV_N)$.}
Fix $\nu\in\Delta_{m,N}$ and define $V_N:\R_+\times\cQ^K\to\R$ by
\begin{equation*}
 V_N(t,a):=\bar F_N^\nu(t,L_Ka+\varepsilon\ell).
\end{equation*}
Write $q_a:=L_Ka+\varepsilon\ell$ and $p_N:=\partial_q\bar F_N^\nu(t,q_a)$. Lemma~\ref{lem:regularity} gives $p_N\in\cQ_\infty$ and $|p_N|_{L^\infty}\le1$. For a direction $b\in E_K$, the chain rule and the definition of $P_K$ give
\begin{equation*}
 \ip{\nabla_KV_N(t,a)}{b}_K
 =\ip{p_N}{L_Kb}_{L^2}
 =\ip{P_Kp_N}{b}_K.
\end{equation*}
Thus $\nabla_KV_N(t,a)=P_Kp_N\in\cQ^K$, and $|\nabla_KV_N(t,a)|_{\ell^\infty}\le1$. In particular, $V_N(t,\cdot)$ is $1$-Lipschitz in $|\cdot|_{\ell^1_K}$ and $(\cQ^K)^*$-increasing.

For $t>0$ and $a\in\operatorname{int}\cQ^K$, Lemma~\ref{lem:step-residual} permits the use of~\eqref{eq:residual} at $q=q_a$. Since $\xi_K(P_Kp_N)=\int_0^1\xi((L_KP_Kp_N)(s))\dd s$ by~\eqref{eq:finite-H}, we have
\begin{align}\label{eq:finite-residual-derivation}
 |\partial_tV_N-\xi_K(\nabla_KV_N)|
 &\le\left|\partial_t\bar F_N^\nu(t,q_a)
 -\int_0^1\xi(p_N(s))\dd s\right|
 +\left|\int_0^1\xi(p_N(s))\dd s-\xi_K(P_Kp_N)\right|\nonumber\\
 &\stackref{\le}{\mathrm{P.\,\ref{prop:residual}},\,\mathrm{L.\,\ref{lem:projection}}}
 C(N\varepsilon)^{-1/3}+\frac{M_1\sqrt D}{K}.
\end{align}

\smallskip
\noindent\emph{Step 2: Comparison with $f^{\nu,K}$.}
Let $\gamma_{N,K,\varepsilon}:=C(N\varepsilon)^{-1/3}+M_1\sqrt D/K$. By~\eqref{eq:finite-residual-derivation}, the functions
\begin{equation*}
 V_N^-(t,a):=V_N(t,a)-\gamma_{N,K,\varepsilon}t
 \quad\text{and}\quad
 V_N^+(t,a):=V_N(t,a)+\gamma_{N,K,\varepsilon}t
\end{equation*}
are respectively a viscosity subsolution and a viscosity supersolution of~\eqref{eq:finite-HJ} on the interior of $\cQ^K$. For each fixed $t\in\R_+$, the functions $V_N^-(t,\cdot)$ and $V_N^+(t,\cdot)$ remain $1$-Lipschitz in $|\cdot|_{\ell^1_K}$ and $(\cQ^K)^*$-increasing. By~\eqref{eq:regularity}, their time increments are bounded by $(M_0+\gamma_{N,K,\varepsilon})|t-t'|$, uniformly in $a$, as required in Lemma~\ref{lem:finite-local-comparison}. The comparison lemma requires these PDE inequalities only on the interior.

For every $b\in\cQ^K$ with $|b|_{\ell^\infty}\le Q'$, the initial condition for $f^{\nu,K}$, Lemma~\ref{lem:regularity}, and $|\ell|_{L^1}=\frac12\sqrt D$ yield
\begin{equation}\label{eq:finite-initial-error}
 |V_N(0,b)-f^{\nu,K}(0,b)|
 \le\sup_{\substack{\nu'\in\Delta_{m,N},\ p\in\cQ_\infty\\|p|_{L^\infty}\le Q'}}|\bar F_N^{\nu'}(0,p)-\varphi^{\nu'}(p)|+\frac12\varepsilon\sqrt D.
\end{equation}
Apply Lemma~\ref{lem:finite-local-comparison} to $(u,v)=(V_N^-,f^{\nu,K})$ and then to $(u,v)=(f^{\nu,K},V_N^+)$. Since $Q'=Q+M_1T+1$, equation~\eqref{eq:finite-initial-error} gives, for $0\le t\le T$ and $|a|_{\ell^\infty}\le Q$,
\begin{equation}\label{eq:comparison-with-VN}
 \begin{aligned}
 |V_N(t,a)-f^{\nu,K}(t,a)|
 &\le \sup_{\substack{\nu'\in\Delta_{m,N},\ p\in\cQ_\infty\\|p|_{L^\infty}\le Q'}}|\bar F_N^{\nu'}(0,p)-\varphi^{\nu'}(p)|\\
 &\quad+\frac12\varepsilon\sqrt D
 +\gamma_{N,K,\varepsilon}t.
 \end{aligned}
\end{equation}
Finally, the $L^1$-Lipschitz estimate~\eqref{eq:regularity} gives
\begin{equation*}
 |\bar F_N^\nu(t,L_Ka)-V_N(t,a)|
 \le\varepsilon|\ell|_{L^1}
 =\frac12\varepsilon\sqrt D.
\end{equation*}
Combining this with~\eqref{eq:comparison-with-VN} proves~\eqref{eq:finite-comparison}.
\end{proof}

We are now ready to rigorously verify that $\bar F_N^\nu \simeq f^\nu$ when $N$ is large. Recall that at the formal level this follows from the fact that $\bar F_N^\nu$ approximately satisfies the same Hamilton--Jacobi equation as $f^\nu$.

\begin{theorem}\label{thm:finite-constrained}
Suppose that $\mu$ has finite support. For every $T,Q\in\R_+$, we have
\begin{equation}\label{eq:finite-constrained}
 \lim_{N\to\infty}
 \sup_{\nu\in\Delta_{m,N}}
 \sup_{\substack{0\le t\le T,\ q\in\cQ_\infty\\|q|_{L^\infty}\le Q}}
 |\bar F_N^\nu(t,q)-f^\nu(t,q)|=0.
\end{equation}
For every $\nu\in\cP(\supp\mu)$ and every sequence $\nu_N\in\Delta_{m,N}$ satisfying $\lim_{N\to\infty}W_1(\nu_N,\nu)=0$, we also have
\begin{equation}\label{eq:finite-constrained-sequential}
 \lim_{N\to\infty}\bar F_N^{\nu_N}(t,q)=f^\nu(t,q),
 \qquad \text{at every}\quad(t,q)\in\R_+\times\cQ_2.
\end{equation}
\end{theorem}
\begin{proof}
Fix $T,Q\in\R_+$ and put $Q':=Q+M_1T+1$. Let $N\in\N$, $\varepsilon>0$, $\eta\in\Delta_{m,N}$, $\nu\in\cP(\supp\mu)$, $t\in[0,T]$, and $q\in\cQ_\infty$ with $|q|_{L^\infty}\le Q$. For each $K$, Lemma~\ref{lem:finite-local-comparison} in both directions, using~\eqref{eq:finite-initial-W1}, gives
\begin{equation*}
 |f^{\eta,K}(t,P_Kq)-f^{\nu,K}(t,P_Kq)|\le2Q'W_1(\eta,\nu).
\end{equation*}
Combine this with Proposition~\ref{prop:finite-comparison}, the $L^1$ regularity of $\bar F_N^\eta$, and~\eqref{eq:projection-L1-bound}:
\begin{align*}
 |\bar F_N^\eta(t,q)-f^\nu(t,q)|
 &\le \sup_{\substack{\nu'\in\Delta_{m,N},\ p\in\cQ_\infty\\|p|_{L^\infty}\le Q'}}|\bar F_N^{\nu'}(0,p)-\varphi^{\nu'}(p)|\\
 &\quad+\varepsilon\sqrt D+C t(N\varepsilon)^{-1/3}
 +2Q'W_1(\eta,\nu)\\
 &\quad+\frac{\sqrt D(Q+tM_1)}K
 +|f^{\nu,K}(t,P_Kq)-f^\nu(t,q)|.
\end{align*}
Take $K\to\infty$ with all other parameters fixed and use~\eqref{eq:HJ-discrete-limit}. We obtain
\begin{equation}\label{eq:finite-constrained-error}
 \begin{aligned}
 |\bar F_N^\eta(t,q)-f^\nu(t,q)|
 &\le \sup_{\substack{\nu'\in\Delta_{m,N},\ p\in\cQ_\infty\\|p|_{L^\infty}\le Q'}}|\bar F_N^{\nu'}(0,p)-\varphi^{\nu'}(p)|\\
 &\quad+\varepsilon\sqrt D+C t(N\varepsilon)^{-1/3}
 +2Q'W_1(\eta,\nu).
 \end{aligned}
\end{equation}
Choose $\varepsilon=N^{-1/4}$. With $\eta=\nu\in\Delta_{m,N}$, the bound tends to zero uniformly over these measures and $t\in[0,T]$, $|q|_{L^\infty}\le Q$, proving~\eqref{eq:finite-constrained}. With $\eta=\nu_N\to\nu$ in $W_1$, it proves~\eqref{eq:finite-constrained-sequential} for bounded paths.

For $q\in\cQ_2$, put $q^{(s)}(r):=q(r\wedge s)$, $s<1$. These paths are bounded and $|q-q^{(s)}|_{L^1}\to0$ as $s\uparrow1$. Equations~\eqref{eq:regularity} and~\eqref{eq:solution-regularity} give
\begin{equation*}
 |\bar F_N^{\nu_N}(t,q)-f^\nu(t,q)|
 \le2|q-q^{(s)}|_{L^1}+|\bar F_N^{\nu_N}(t,q^{(s)})-f^\nu(t,q^{(s)})|.
\end{equation*}
Let $N\to\infty$ and then $s\uparrow1$ to finish the sequential assertion.
\end{proof}

To convert the previous theorem into a proof of \eqref{eq:main}, we need one last result. Namely, we need to verify that $f^\nu$ is continuous in $\nu$. This way we can transfer our control for $\nu \in \Delta_{m,N}$ to a control for every $\nu \in \cP(\supp \mu)$.
\begin{lemma}\label{lem:finite-measure-continuity}
Suppose that $\mu$ has finite support. For every $\nu,\nu'\in\cP(\supp\mu)$, $t\in\R_+$, and $q\in\cQ_\infty$,
\begin{equation*}
 |f^\nu(t,q)-f^{\nu'}(t,q)|\le2(tM_1+|q|_{L^\infty})W_1(\nu,\nu').
\end{equation*}
\end{lemma}
\begin{proof}
Choose $\nu_N,\nu_N'\in\Delta_{m,N}$ converging to $\nu,\nu'$ in $W_1$. Lemma~\ref{lem:type-continuity} gives the same bound for $\bar F_N^{\nu_N}$ and $\bar F_N^{\nu_N'}$. Pass to the limit using~\eqref{eq:finite-constrained-sequential} and continuity of $W_1$. For any two finitely supported target measures, apply Theorem~\ref{thm:finite-constrained} with a reference probability measure whose support is the union of their supports; the initial data and solutions depend only on the target measure.
\end{proof}

Combining all the previous results of the section, we obtain the following weak-version of the main result. This is only a weaker version because it assumes that $\mu$ has a finite support. We explain in the next section how to deduce the main result from this weak version.

\begin{corollary}\label{cor:finite-main}
If $\mu$ has finite support, then~\eqref{eq:main} holds. For every $T,Q\in\R_+$, the convergence is uniform over $t\in[0,T]$ and $q\in\mcl Q_{\infty,\leq Q}$.
\end{corollary}
\begin{proof}
Fix $T,Q\in\R_+$ and assume $\supp\mu=\{a_1,\dots,a_m\}$. We identify the limit uniformly over $t\in[0,T]$ and $q\in \mcl Q_{\infty,\leq Q}$.

\smallskip
\noindent\emph{Step 1: Reduction to a minimum over $\Delta_{m,N}$.}
Let $s_{0,N}:=\min\{1/2,\sqrt{\log(N+1)/N}\}$. Then we have
\begin{equation*}
 \lim_{N\to\infty}s_{0,N}=0
 \qquad\text{and}\qquad
 \lim_{N\to\infty}\frac{\log(N+1)}{Ns_{0,N}}=0.
\end{equation*}
By Lemma~\ref{lem:type-minimum}, the difference between $\bar F_{N,\mu^{\otimes N}}(t,q)$ and the minimum of $\bar F_{N,\mu^{\otimes N},\Sigma_N(\nu)}(t,q)$ is between zero and $Cs_{0,N}+m\log(N+1)/(Ns_{0,N})$. By~\eqref{eq:normalization} and~\eqref{eq:stirling}, we therefore have
\begin{align*}
 &\left|\bar F_{N,\mu^{\otimes N}}(t,q)
 -\min_{\nu\in\Delta_{m,N}}\{\Ent(\nu\mid\mu)+f^\nu(t,q)\}\right|\\
 &\quad\le Cs_{0,N}+\frac{m\log(N+1)}{Ns_{0,N}}
 +\frac{m\log(N+1)}{N}
 +\sup_{\nu\in\Delta_{m,N}}|\bar F_N^\nu(t,q)-f^\nu(t,q)|.
\end{align*}
By \eqref{eq:finite-constrained} in Theorem~\ref{thm:finite-constrained}, the right-hand side tends to zero uniformly over $t\in[0,T]$ and $q\in \mcl Q_{\infty,\leq Q}$.

\smallskip
\noindent\emph{Step 2: Passing from $\Delta_{m,N}$ to $\cP(\supp\mu)$.}
For $t\in[0,T]$ and $q\in \mcl Q_{\infty,\leq Q}$, define $G_{t,q}:\cP(\supp\mu)\to\R$ by
\begin{equation*}
 G_{t,q}(\nu):=\Ent(\nu\mid\mu)+f^\nu(t,q).
\end{equation*}
By Lemma~\ref{lem:finite-measure-continuity}, we have
\begin{equation*}
 |G_{t,q}(\nu)-G_{t,q}(\nu')|
 \le|\Ent(\nu\mid\mu)-\Ent(\nu'\mid\mu)|
 +2(Q+M_1T)W_1(\nu,\nu').
\end{equation*}
The entropy is uniformly continuous on the compact set $\cP(\supp\mu)$. Define $\nu^{[N]}:\cP(\supp\mu)\to\Delta_{m,N}$ by taking $\lfloor N\nu_j\rfloor$ spins equal to $a_j$ for $j\in\{1,\ldots,m-1\}$ and assigning the remaining spins to $a_m$. For every $\nu\in\cP(\supp\mu)$, we then have
\begin{equation*}
 \sum_{j=1}^m|\nu_j-\nu^{[N]}(\nu)(\{a_j\})|\le\frac{2(m-1)}N.
\end{equation*}
Since the total variation distance between two measures on $\supp\mu$ is half the sum of the differences of their atom masses, a coupling gives
\begin{equation*}
 \sup_{\nu\in\cP(\supp\mu)}W_1\bigl(\nu,\nu^{[N]}(\nu)\bigr)
 \le
 \frac1N
 (m-1)\diam(\supp\mu).
\end{equation*}
Since $\Delta_{m,N}\subset\cP(\supp\mu)$, the preceding two estimates give
\begin{equation*}
 \lim_{N\to\infty}\sup_{\substack{0\le t\le T\\q\in \mcl Q_{\infty,\leq Q}}}
 \left|\min_{\nu\in\Delta_{m,N}}G_{t,q}(\nu)-\min_{\nu\in\cP(\supp\mu)}G_{t,q}(\nu)\right|=0.
\end{equation*}
Combining Steps 1 and 2 proves~\eqref{eq:main} and its uniformity for the finite-support case. For $q\in\cQ_2$, the conclusion follows by approximating $q$ in $L^1$ by bounded paths and using~\eqref{eq:regularity} and~\eqref{eq:solution-regularity}.
\end{proof}

\section{Extension to every reference measure on the unit ball} \label{sec:ext-unit-ball}

Throughout this section, $\mu$ is an arbitrary probability measure supported on $\B$. The initial data and solutions were defined in~\eqref{eq:initial-definition} and Definition~\ref{def:spin-glass-solution}. We approximate them and the reference measure by finite discretizations to prove Theorem~\ref{thm:main}.

\begin{proposition}\label{prop:general-initial}
For every $\nu,\nu'\in\cP(\supp\mu)$, $t\in\R_+$, and $q\in\cQ_\infty$, the solutions of Definition~\ref{def:spin-glass-solution} satisfy
\begin{equation}\label{eq:measure-continuity}
 |f^\nu(t,q)-f^{\nu'}(t,q)|\le2(tM_1+|q|_{L^\infty})W_1(\nu,\nu').
\end{equation}
In addition, for every $(t,q) \in\R_+ \times \cQ_2$, the map $\nu\mapsto f^\nu(t,q)$ is continuous in $W_1$.
\end{proposition}
\begin{proof}
The regularity of every $\varphi^\nu$ follows from Lemma~\ref{lem:initial-regularity}. Since $\supp\mu$ is compact, choose for each $\delta>0$ a finite $\delta$-net and a Borel nearest-point map $\mcl T_\delta:\supp\mu\to\supp\mu$. Then
\begin{equation}\label{eq:discretization-distance}
 |\mcl T_\delta(a)-a|\le\delta,\qquad a\in\supp\mu.
\end{equation}
Put $\nu_\delta:=(\mcl T_\delta)_\#\nu$, so $W_1(\nu_\delta,\nu)\le\delta$.

Fix $q\in\cQ_\infty$ and $\delta>0$. Write $\mcl T_\delta(\supp\mu)=\{a_1,\ldots,a_m\}$ and let $C_j:=\mcl T_\delta^{-1}(\{a_j\})$ for $j\in\{1,\ldots,m\}$. Fix $h\in\bb(\supp\mu)$, and define $h_\delta:\{a_1,\ldots,a_m\}\to\R$ by
\begin{equation*}
 h_\delta(a_j):=\begin{cases}
 \displaystyle\log\left(\frac1{\nu(C_j)}\int_{C_j}e^{h(a)}\,\nu(\mathrm da)\right),&\nu(C_j)>0,\\
 0,&\nu(C_j)=0.
 \end{cases}
\end{equation*}
Extend $h_\delta$ to a bounded Borel function on $\supp\mu$. By construction, we have
\begin{equation*}
 (\mcl T_\delta)_\#(e^h\nu)=e^{h_\delta}\nu_\delta
 \quad\text{and}\quad
 \int e^h\,\mathrm d\nu=\int e^{h_\delta}\,\mathrm d\nu_\delta.
\end{equation*}
Recall from~\eqref{eq:initial-definition} that $\psi_{\nu,h}(q)=\bar F_{1,e^h\nu}(0,q)$ and $\psi_{\nu_\delta,h_\delta}(q)=\bar F_{1,e^{h_\delta}\nu_\delta}(0,q)$. The two finite reference measures $e^h\nu$ and $e^{h_\delta}\nu_\delta$ have a common positive mass. Normalize them by this mass and couple the resulting probability measures through $\mcl T_\delta$. Lemma~\ref{lem:coupling} with $N=1$ and $t=0$ bounds the difference of their free energies; restoring the common mass changes each free energy by the same constant. Hence, we have
\begin{equation}\label{eq:tilt-discretization}
 |\psi_{\nu,h}(q)-\psi_{\nu_\delta,h_\delta}(q)|
 \le2|q|_{L^\infty}\delta.
\end{equation}
Jensen's inequality on each $C_j$ gives
\begin{equation*}
 \int_{C_j}h\,\mathrm d\nu\le\nu(C_j)h_\delta(a_j),
 \qquad j\in\{1,\ldots,m\}.
\end{equation*}
Combining this inequality with~\eqref{eq:tilt-discretization} and taking the supremum over $h$ in~\eqref{eq:initial-definition}, we obtain
\begin{equation*}
 \varphi^\nu(q)\le\varphi^{\nu_\delta}(q)+2|q|_{L^\infty}\delta.
\end{equation*}
Conversely, for any bounded $h_\delta$ on $\mcl T_\delta(\supp\mu)$, choose $h:=h_\delta\circ\mcl T_\delta$. In this case $\int h\,\mathrm d\nu=\int h_\delta\,\mathrm d\nu_\delta$, and~\eqref{eq:tilt-discretization} applies. Taking the supremum over $h_\delta$ gives the reverse inequality, and therefore
\begin{equation}\label{eq:initial-discretization}
 |\varphi^\nu(q)-\varphi^{\nu_\delta}(q)|
 \le2|q|_{L^\infty}\delta,
 \qquad q\in\cQ_\infty.
\end{equation}

For $b\in\cQ^K$,~\eqref{eq:initial-discretization} gives
\begin{equation*}
 |\varphi^\nu(L_Kb)-\varphi^{\nu_\delta}(L_Kb)|\le2|b|_{\ell^\infty}\delta.
\end{equation*}
Apply Lemma~\ref{lem:finite-local-comparison} in both directions to the already-defined solutions $f^{\nu,K}$ and $f^{\nu_\delta,K}$. For $t\le T$ and $|a|_{\ell^\infty}\le Q$,
\begin{equation}\label{eq:general-finite-comparison}
 |f^{\nu,K}(t,a)-f^{\nu_\delta,K}(t,a)|\le2(Q+M_1T+1)\delta.
\end{equation}
Taking $a=P_Kq$ and letting $K\to\infty$ in~\eqref{eq:HJ-discrete-limit}, we can deduce from the previous display that $f^{\nu_\delta}\to f^\nu$ uniformly on $[0,T]\times\mcl Q_{\infty,\le Q}$. This means that if we approximate the measures $\nu$ and $\nu'$ by $\nu_\delta$ and $\nu'_{\delta'}$, apply Lemma~\ref{lem:finite-measure-continuity}, and let $\delta,\delta'\downarrow0$, we obtain~\eqref{eq:measure-continuity}.

For $q\in\cQ_2$, put $q^{(s)}(r):=q(r\wedge s)$ for $s<1$. If $|q|_{L^2}\le R$, monotonicity and the Cauchy--Schwarz inequality give
\begin{equation}\label{eq:uniform-truncation}
 |q^{(s)}|_{L^\infty}\le\frac R{\sqrt{1-s}},
 \qquad |q-q^{(s)}|_{L^1}\le2R\sqrt{1-s}.
\end{equation}
Indeed, $(1-s)|q(s)|^2\le\int_s^1|q(r)|^2\dd r$, and $|q-q^{(s)}|_{L^1}\le\int_s^1|q(r)|\dd r+(1-s)|q(s)|$. Apply~\eqref{eq:solution-regularity} at $q$ and $q^{(s)}$, and~\eqref{eq:measure-continuity} at the bounded path. We obtain, for all target measures,
\begin{equation}\label{eq:L2-measure-continuity}
 \sup_{\substack{0\le t\le T,\ q\in\cQ_2\\|q|_{L^2}\le R}}
 |f^\nu(t,q)-f^{\nu'}(t,q)|
 \le4R\sqrt{1-s}+2\left(TM_1+\frac R{\sqrt{1-s}}\right)W_1(\nu,\nu').
\end{equation}
First choose $s$ close to $1$ and then let $W_1(\nu,\nu')\to0$. This proves continuity and also uniform convergence of the finite discretizations on bounded subsets of $\cQ_2$. 
\end{proof}

For the remainder of the section, fix $\delta>0$ and the map $\mcl T_\delta$ satisfying~\eqref{eq:discretization-distance}, and let $\mu_\delta:=(\mcl T_\delta)_\#\mu$. Its support is finite and contained in $\supp\mu$. Applying Lemma~\ref{lem:coupling} to the product coupling induced by $\mcl T_\delta$ gives, for $t\in\R_+$ and $q\in\cQ_\infty$,
\begin{equation}\label{eq:reference-discretization}
 \sup_{N\in\N}|\bar F_{N,\mu^{\otimes N}}(t,q)-\bar F_{N,\mu_\delta^{\otimes N}}(t,q)|
 \le2\delta\bigl(tM_1+|q|_{L^\infty}\bigr).
\end{equation}
Indeed, the joint law of $(\sigma,\mcl T_\delta(\sigma))$ under $\mu^{\otimes N}$ has marginals $\mu^{\otimes N}$ and $\mu_\delta^{\otimes N}$, and every pair in its support satisfies $N^{-1}\sum_{i=1}^N|\sigma_{\cdot,i}-\mcl T_\delta(\sigma_{\cdot,i})|\le\delta$. Thus~\eqref{eq:reference-discretization} is independent of the number of points in $\supp\mu_\delta$.

Recall the relative entropy $\Ent(\cdot\mid\mu)$ introduced in \eqref{eq:relative-entropy}. Define $J,J_\delta:\R_+\times\cQ_2\to\R$ by
\begin{align}\label{eq:envelopes}
 J(t,q):=\inf_{\nu\in\cP(\supp\mu)}\bigl\{\Ent(\nu\mid\mu)+f^\nu(t,q)\bigr\}\quad\text{and}\quad
 J_\delta(t,q):=\inf_{\eta\in\cP(\supp\mu_\delta)}\bigl\{\Ent(\eta\mid\mu_\delta)+f^\eta(t,q)\bigr\}.
\end{align}
Both quantities in~\eqref{eq:envelopes} are finite. Indeed, entropy is nonnegative, and~\eqref{eq:solution-regularity} and Lemma~\ref{lem:initial-regularity} give $|f^\nu(t,q)|\le M_0t+|q|_{L^1}$ uniformly in $\nu$; the choices $\nu=\mu$ and $\eta=\mu_\delta$ have zero entropy.

\begin{lemma}\label{lem:envelope-discretization}
For every $t\in\R_+$ and $q\in\cQ_\infty$, we have
\begin{equation}\label{eq:envelope-discretization}
 |J(t,q)-J_\delta(t,q)|\le2\bigl(tM_1+|q|_{L^\infty}\bigr)\delta.
\end{equation}
\end{lemma}
\begin{proof}
Write $\supp\mu_\delta=\{a_1,\ldots,a_m\}$ and define the cells $C_j:=\mcl T_\delta^{-1}(\{a_j\})$. Then $\mu(C_j)=\mu_\delta(\{a_j\})>0$, and these cells cover $\supp\mu$ up to a $\mu$-null set.

\smallskip
\noindent\emph{Entropy contraction.}
Fix $\nu\in\cP(\supp\mu)$ with $\Ent(\nu\mid\mu)<\infty$ and let $\eta:=(\mcl T_\delta)_\#\nu$. 
The coupling $(a,\mcl T_\delta(a))$ gives $W_1(\nu,\eta)\le\delta$. In addition, relative entropy decreases under the measurable map $\mcl T_\delta$:
\begin{equation}\label{eq:entropy-contraction}
 \Ent(\eta\mid\mu_\delta)\le\Ent(\nu\mid\mu).
\end{equation}
Indeed, write $r:=\mathrm d\nu/\mathrm d\mu$. On every cell with $\mu(C_j)>0$, Jensen's inequality for $x\mapsto x\log x$ gives
\begin{equation*}
 \int_{C_j}r\log r\,\mathrm d\mu
 \ge\nu(C_j)\log\frac{\nu(C_j)}{\mu(C_j)}.
\end{equation*}
Summing these inequalities proves~\eqref{eq:entropy-contraction}. Combining it with~\eqref{eq:measure-continuity}, we obtain
\begin{align*}
 \Ent(\eta\mid\mu_\delta)+f^\eta(t,q)
 &\le\Ent(\nu\mid\mu)+f^\nu(t,q)
 +2\bigl(tM_1+|q|_{L^\infty}\bigr)\delta.
\end{align*}
Taking the infimum over such $\nu$ proves the first inequality.

\smallskip
\noindent\emph{Entropy-preserving lift.}
Fix $\eta\in\cP(\supp\mu_\delta)$ and let $\beta_j:=\eta(\{a_j\})$. Define $\widetilde\nu\in\cP(\supp\mu)$ by
\begin{equation}\label{eq:entropy-preserving-lift}
 \widetilde\nu:=\sum_{j=1}^m\beta_j\,\mu(\,\cdot\mid C_j).
\end{equation}
The pushforward $(\mcl T_\delta)_\#\widetilde\nu$ equals $\eta$, so $W_1(\widetilde\nu,\eta)\le\delta$. Moreover, the density of the measure in~\eqref{eq:entropy-preserving-lift} with respect to $\mu$ equals $\beta_j/\mu(C_j)$ on $C_j$. Consequently, we have
\begin{equation}\label{eq:entropy-preserving-identity}
 \Ent(\widetilde\nu\mid\mu)
 =\sum_{j=1}^m\beta_j\log\frac{\beta_j}{\mu(C_j)}
 =\Ent(\eta\mid\mu_\delta).
\end{equation}
Using~\eqref{eq:entropy-preserving-identity} and~\eqref{eq:measure-continuity}, we therefore obtain
\begin{align*}
 \Ent(\widetilde\nu\mid\mu)+f^{\widetilde\nu}(t,q)
 &\le\Ent(\eta\mid\mu_\delta)+f^\eta(t,q)
 +2\bigl(tM_1+|q|_{L^\infty}\bigr)\delta.
\end{align*}
Taking the infimum over $\eta$ completes~\eqref{eq:envelope-discretization}.
\end{proof}

\begin{proof}[Proof of Theorem~\ref{thm:main}]
\smallskip
\noindent\emph{Step 1: Uniform convergence on bounded paths.}
Fix $T,Q\in\R_+$. Since $\mu_\delta$ has finite support, Corollary~\ref{cor:finite-main} gives
\begin{equation}\label{eq:quantized-free-energy-limit}
 \lim_{N\to\infty}\sup_{\substack{0\le t\le T,\ q\in\mcl Q_{\infty,\le Q}}}
 |\bar F_{N,\mu_\delta^{\otimes N}}(t,q)-J_\delta(t,q)|=0.
\end{equation}
For every $(t,q)\in[0,T]\times\mcl Q_{\infty,\le Q}$, the triangle inequality gives
\begin{align*}
 |\bar F_{N,\mu^{\otimes N}}(t,q)-J(t,q)|
 &\le|\bar F_{N,\mu^{\otimes N}}(t,q)-\bar F_{N,\mu_\delta^{\otimes N}}(t,q)|\\
 &\quad+|\bar F_{N,\mu_\delta^{\otimes N}}(t,q)-J_\delta(t,q)|
 +|J_\delta(t,q)-J(t,q)|.
\end{align*}
By~\eqref{eq:reference-discretization}, Lemma~\ref{lem:envelope-discretization}, and~\eqref{eq:quantized-free-energy-limit}, we therefore obtain
\begin{equation}\label{eq:general-uniform-error}
 \limsup_{N\to\infty}\sup_{\substack{0\le t\le T,\ q\in\mcl Q_{\infty,\le Q}}}
 |\bar F_{N,\mu^{\otimes N}}(t,q)-J(t,q)|
 \le4\delta(TM_1+Q).
\end{equation}
Since $\delta>0$ is arbitrary, sending $\delta\downarrow0$ in~\eqref{eq:general-uniform-error} proves uniform convergence on $[0,T]\times\mcl Q_{\infty,\le Q}$.

\smallskip
\noindent\emph{Step 2: Extension to every $q\in\cQ_2$.}
Fix $t\in\R_+$ and $q\in\cQ_2$. For $s\in[0,1)$, take $q^{(s)}$ from~\eqref{eq:uniform-truncation}. Lemma~\ref{lem:regularity} and~\eqref{eq:solution-regularity} imply that both $q\mapsto\bar F_{N,\mu^{\otimes N}}(t,q)$ and $q\mapsto J(t,q)$ are $1$-Lipschitz in $L^1$; for $J$, this follows by taking the infimum of functions with the same Lipschitz constant. Hence
\begin{align*}
 |\bar F_{N,\mu^{\otimes N}}(t,q)-J(t,q)|
 &\le2|q-q^{(s)}|_{L^1}
 +|\bar F_{N,\mu^{\otimes N}}(t,q^{(s)})-J(t,q^{(s)})|.
\end{align*}
The second term tends to zero as $N\to\infty$ for fixed $s$, by Step 1. By~\eqref{eq:uniform-truncation}, the first term tends to zero as $s\uparrow1$. This proves~\eqref{eq:main} with the infimum in place of the minimum.

\smallskip
\noindent\emph{Step 3: Attainment of the infimum.}
The space $\cP(\supp\mu)$ is compact in $W_1$, since $\supp\mu$ is compact. Relative entropy is lower semicontinuous in $W_1$. By~\eqref{eq:L2-measure-continuity}, the map $\nu\mapsto f^\nu(t,q)$ is continuous in $W_1$ for every $q\in\cQ_2$. Therefore, the objective
\begin{equation*}
 \nu\longmapsto\Ent(\nu\mid\mu)+f^\nu(t,q)
\end{equation*}
is lower semicontinuous on $\cP(\supp\mu)$. It is finite at $\nu=\mu$, since $\Ent(\mu\mid\mu)=0$ and $f^\mu(t,q)$ is finite. Compactness therefore yields a minimizer, proving~\eqref{eq:main} as stated.
\end{proof}

\appendix
\section{Proof of the finite-dimensional comparison principle}\label{app:finite-local-comparison}

We give a proof of Lemma~\ref{lem:finite-local-comparison} using the weighted norms of Section~\ref{sec:constrained-HJ}. In particular, the argument uses only the monotonicity and regularity of $\xi_K$ on $\cQ^K$ and does not require a boundary condition on $\partial\cQ^K$.

\begin{proof}[Proof of Lemma~\ref{lem:finite-local-comparison}]
Fix $K\in\N$. Use the weighted inner product $\ip\cdot\cdot_K$ and let $|a|_K:=\ip a a_K^{1/2}$. Then
\begin{equation*}
 |a|_{\ell^1_K}\le|a|_K\le|a|_{\ell^\infty},
 \qquad a\in E_K,
\end{equation*}
and $|\cdot|_{\ell^1_K}$ and $|\cdot|_{\ell^\infty}$ are dual norms. By assumption, for some $C\in\R_+$,
\begin{equation}\label{eq:app-time-growth}
 |u(t,a)-u(0,a)|\vee|v(t,a)-v(0,a)|\le Ct,
 \qquad(t,a)\in\R_+\times\cQ^K.
\end{equation}

\smallskip
\noindent\emph{Localization and the boundary barrier.}
Put $R:=Q+M_1T+1$ and
\begin{equation*}
 S_R:=\sup_{a\in\cQ^K:\,|a|_{\ell^\infty}\le R}(u(0,a)-v(0,a)).
\end{equation*}
Radial projection onto this ball, together with the spatial Lipschitz bounds, gives
\begin{equation}\label{eq:app-initial-radial}
 u(0,a)-v(0,a)\le S_R+2(|a|_{\ell^\infty}-R)_+,
 \qquad a\in\cQ^K.
\end{equation}
For $\kappa\in(0,1/4)$, define the smooth function
\begin{equation}\label{eq:app-smoothed-norm}
 \rho_\kappa(a):=\sqrt{|a_K|^2+\kappa^2}.
\end{equation}
On $\cQ^K$, positivity of the increments implies $|a|_{\ell^\infty}=|a_K|$, so
\begin{equation*}
 |a|_{\ell^\infty}\le\rho_\kappa(a)\le|a|_{\ell^\infty}+\kappa,
 \qquad a\in\cQ^K.
\end{equation*}
Moreover, $\nabla_K\rho_\kappa$ has only its last block nonzero, equal to $Ka_K/\rho_\kappa(a)$, and hence $|\nabla_K\rho_\kappa|_{\ell^1_K}\le1$ on $E_K$. Choose a smooth nondecreasing $\theta_\kappa:\R\to\R_+$ with $0\le\theta_\kappa'\le1$, $\theta_\kappa(r)=0$ for $r\le0$, and $\theta_\kappa(r)\ge(r-\kappa)_+$. Let
\begin{equation}\label{eq:app-local-cutoff}
 \chi(t,a):=3\theta_\kappa(\rho_\kappa(a)+M_1t-R).
\end{equation}
The preceding bounds give
\begin{align}\label{eq:app-cutoff-properties}
 \partial_t\chi&\ge M_1|\nabla_K\chi|_{\ell^1_K},\nonumber\\
 \chi(t,a)&=0,\qquad 0\le t\le T,\quad a\in\cQ^K,\quad |a|_{\ell^\infty}\le Q,\nonumber\\
 \chi(t,a)&\ge3(|a|_{\ell^\infty}-R-2\kappa)_+,
 \qquad(t,a)\in\R_+\times\cQ^K.
\end{align}
In particular,
\begin{equation}\label{eq:app-initial-cutoff}
 \sup_{a\in\cQ^K}(u(0,a)-v(0,a)-\chi(0,a))\le S_R+6\kappa.
\end{equation}

For $a\in\operatorname{int}\cQ^K$, put $a_0:=0$ and define
\begin{equation}\label{eq:app-boundary-function}
 b(a):=\sum_{k=1}^K\tr[(a_k-a_{k-1})^{-1}].
\end{equation}
This smooth positive function diverges at the boundary. For $c\in\cQ^K$, differentiation gives
\begin{equation}\label{eq:app-boundary-monotone}
 \ip{\nabla_Kb(a)}c_K
 =-\sum_{k=1}^K\tr[(a_k-a_{k-1})^{-2}(c_k-c_{k-1})]\le0,
 \qquad c_0:=0.
\end{equation}
Thus $-\nabla_Kb(a)\in(\cQ^K)^*$.

\smallskip
\noindent\emph{Doubled variables and interior contact points.}
Suppose the comparison fails. There are $t_*\in(0,T]$ and $a_*\in\cQ^K$, with $|a_*|_{\ell^\infty}\le Q$, such that $u(t_*,a_*)-v(t_*,a_*)>S_R$. Fix $T_*>T$. Choose $\kappa$, and then $\gamma,\beta>0$, small enough that
\begin{equation*}
 u(t_*,a_*)-v(t_*,a_*)-\chi(t_*,a_*)-\gamma t_*
 -\frac{2\beta}{T_*-t_*}>S_R+12\kappa.
\end{equation*}
By continuity and density of $\operatorname{int}\cQ^K$, choose $\widehat a_*\in\operatorname{int}\cQ^K$ close to $a_*$, and then $\sigma>0$ small enough that
\begin{equation}\label{eq:app-diagonal-gap}
 u(t_*,\widehat a_*)-v(t_*,\widehat a_*)-\chi(t_*,\widehat a_*)-\gamma t_*
 -\frac{2\beta}{T_*-t_*}-\sigma b(\widehat a_*)>S_R+12\kappa.
\end{equation}
Keep these parameters fixed. For $\alpha>1$, define on $[0,T_*)^2\times\operatorname{int}\cQ^K\times\cQ^K$
\begin{align}\label{eq:app-doubled-function}
 \Psi_\alpha(t,t',a,a')
 &:=u(t,a)-v(t',a')-\frac\alpha2(|t-t'|^2+|a-a'|_K^2)\nonumber\\
 &\quad-\chi(t,a)-\gamma t-\frac\beta{T_*-t}-\frac\beta{T_*-t'}-\sigma b(a).
\end{align}
The diagonal point from~\eqref{eq:app-diagonal-gap} gives
\begin{equation}\label{eq:app-max-lower}
 \sup\Psi_\alpha>S_R+12\kappa.
\end{equation}
The time-growth and spatial Lipschitz assumptions imply
\begin{equation}\label{eq:app-growth-bound}
 u(t,a)-v(t',a')\le S_R+2(|a|_{\ell^\infty}-R)_++2CT_*+|a-a'|_{\ell^1_K}.
\end{equation}
The cutoff with coefficient $3>2$ controls $a$, the quadratic penalty controls $a'$, the time penalties exclude $T_*$, and $\sigma b(a)$ excludes $\partial\cQ^K$. Hence $\Psi_\alpha$ attains a maximum at $(t_\alpha,t'_\alpha,a_\alpha,a'_\alpha)$. Equations~\eqref{eq:app-growth-bound} and~\eqref{eq:app-max-lower}, using $x-\alpha x^2/2\le1/(2\alpha)$, give a constant $C_*<\infty$, independent of $\alpha>1$, such that
\begin{equation}\label{eq:app-compactness}
 |a_\alpha|_K\vee|a'_\alpha|_K\le C_*,
 \qquad
 \alpha(|t_\alpha-t'_\alpha|^2+|a_\alpha-a'_\alpha|_K^2)+b(a_\alpha)\le C_*.
\end{equation}
Since each inverse trace is bounded by $C_*$, every increment of $a_\alpha$ has smallest eigenvalue at least $C_*^{-1}$. Quadratic closeness gives $|a_\alpha-a'_\alpha|_K\to0$; for fixed $K$, Weyl's inequality then shows $a'_\alpha\in\operatorname{int}\cQ^K$ for all sufficiently large $\alpha$.

If either contact time is zero, the time-growth assumption gives
\begin{equation*}
 u(t_\alpha,a_\alpha)-v(t'_\alpha,a'_\alpha)
 \le u(0,a_\alpha)-v(0,a_\alpha)+C|t_\alpha-t'_\alpha|+|a_\alpha-a'_\alpha|_K.
\end{equation*}
Since $\chi(t_\alpha,a_\alpha)\ge\chi(0,a_\alpha)$, completing squares and using~\eqref{eq:app-initial-cutoff} yields
\begin{equation*}
 \Psi_\alpha(t_\alpha,t'_\alpha,a_\alpha,a'_\alpha)
 \le S_R+6\kappa+\frac{C^2+1}{2\alpha},
\end{equation*}
contradicting~\eqref{eq:app-max-lower} for large $\alpha$. Fix such an $\alpha$, so both contact points lie in $(0,T_*)\times\operatorname{int}\cQ^K$.

\smallskip
\noindent\emph{Hamiltonian comparison and viscosity inequalities.}
The test gradients for $u$ and $v$ are
\begin{equation}\label{eq:app-contact-gradients}
 p^v:=\alpha(a_\alpha-a'_\alpha),
 \qquad
 p^u:=p^v+\nabla_K\chi(t_\alpha,a_\alpha)+\sigma\nabla_Kb(a_\alpha).
\end{equation}
At interior contact points, the spatial Lipschitz and $(\cQ^K)^*$-monotonicity assumptions give
\begin{equation}\label{eq:app-gradient-cone}
 p^u,p^v\in\cQ^K,
 \qquad |p^u|_{\ell^\infty}\vee|p^v|_{\ell^\infty}\le1.
\end{equation}
For example, test monotonicity in a direction $c\in(\cQ^K)^*$ at $a_\alpha+hc$ for the upper test and at $a'_\alpha-hc$ for the lower test. Letting $h\downarrow0$ gives nonnegative pairing with every $c\in(\cQ^K)^*$, hence membership in $((\cQ^K)^*)^*=\cQ^K$. Testing arbitrary directions gives the norm bounds.

The monotonicity of $\xi_K$ and the definition of $M_1$ imply
\begin{equation}\label{eq:app-h-gradient}
 \nabla_K\xi_K(p)\in\cQ^K,
 \qquad |\nabla_K\xi_K(p)|_{\ell^\infty}\le M_1,
 \qquad p\in\cQ^K,\quad |p|_{\ell^\infty}\le1.
\end{equation}
The first inclusion follows by differentiating monotonicity in directions of $(\cQ^K)^*$ at interior points and extending by continuity; the norm bound follows from $(\nabla_K\xi_K(p))_k=\nabla\xi(p_k)$. By~\eqref{eq:app-boundary-monotone},
\begin{equation}\label{eq:app-boundary-dual}
 r:=-\sigma\nabla_Kb(a_\alpha)\in(\cQ^K)^*,
 \qquad p^u-p^v=\nabla_K\chi(t_\alpha,a_\alpha)-r.
\end{equation}
Integrating the derivative of $\xi_K$ along the segment between $p^v$ and $p^u$, which lies in the set in~\eqref{eq:app-gradient-cone}, gives
\begin{align}\label{eq:app-h-difference}
 \xi_K(p^u)-\xi_K(p^v)
 &=\int_0^1\ip{\nabla_K\xi_K(p^v+s(p^u-p^v))}{\nabla_K\chi-r}_K\dd s\nonumber\\
 &\le M_1|\nabla_K\chi(t_\alpha,a_\alpha)|_{\ell^1_K}.
\end{align}
The time derivatives of the upper and lower tests are respectively
\begin{equation*}
 \alpha(t_\alpha-t'_\alpha)+\partial_t\chi(t_\alpha,a_\alpha)+\gamma
 +\frac\beta{(T_*-t_\alpha)^2},
 \qquad
 \alpha(t_\alpha-t'_\alpha)-\frac\beta{(T_*-t'_\alpha)^2}.
\end{equation*}
Subtracting the viscosity supersolution inequality from the subsolution inequality, and discarding the positive time-penalty terms, yields
\begin{equation*}
 \gamma\le \xi_K(p^u)-\xi_K(p^v)-\partial_t\chi(t_\alpha,a_\alpha)
 \le M_1|\nabla_K\chi|_{\ell^1_K}-\partial_t\chi\le0,
\end{equation*}
by~\eqref{eq:app-cutoff-properties}. This contradicts $\gamma>0$ and proves the comparison.
\end{proof}

\small
\bibliographystyle{plain}
\bibliography{ref}

@preamble{"\newcommand{\noop}[1]{} "}

@preamble{"\def\cprime{$'$} "}

@article{camilli2026perceptron,
  author        = {Camilli, Francesco and Contucci, Pierluigi and Gerace, Federica and Mingione, Emanuele},
  title         = {Variational bounds for perceptron learning from structured data},
  journal       = {Preprint, arXiv:2608.04882},
  year          = {2026},
  eprint        = {2608.04882},
  archivePrefix = {arXiv}
}

@article {barcon,
    AUTHOR = {Barra, Adriano and Contucci, Pierluigi and Mingione, Emanuele
              and Tantari, Daniele},
     TITLE = {Multi-species mean field spin glasses. {R}igorous results},
   JOURNAL = {Ann. Henri Poincar\'{e}},
    VOLUME = {16},
      YEAR = {2015},
    NUMBER = {3},
     PAGES = {691--708},
      ISSN = {1424-0637},
   MRCLASS = {82D30 (80A17)},
  MRNUMBER = {3311887},
       DOI = {10.1007/s00023-014-0341-5},
       URL = {https://doi.org/10.1007/s00023-014-0341-5},
}

@article {barra1,
    AUTHOR = {Barra, Adriano and Di Biasio, Aldo and Guerra, Francesco},
     TITLE = {Replica symmetry breaking in mean-field spin glasses through
              the {Hamilton--Jacobi} technique},
   JOURNAL = {J. Stat. Mech. Theory Exp.},
      YEAR = {2010},
     PAGES = {P09006, 22},
      ISSN = {1742-5468},
   MRCLASS = {82B44 (82D30)},
  MRNUMBER = {2727596},
MRREVIEWER = {Aernout C. D. van Enter},
}

@article {abarra,
    AUTHOR = {Agliari, Elena and Barra, Adriano and Burioni, Raffaella and
              Di Biasio, Aldo},
     TITLE = {Notes on the p-spin glass studied via {Hamilton--Jacobi} and
              smooth-cavity techniques},
   JOURNAL = {J. Math. Phys.},
    VOLUME = {53},
      YEAR = {2012},
    NUMBER = {6},
     PAGES = {063304, 29},
      ISSN = {0022-2488},
   MRCLASS = {82B44 (82D30)},
  MRNUMBER = {2977685},
MRREVIEWER = {Mei Yin},
       DOI = {10.1063/1.4729233},
       URL = {https://doi.org/10.1063/1.4729233},
}

@article {barra2,
    AUTHOR = {Barra, Adriano and Del Ferraro, Gino and Tantari, Daniele},
     TITLE = {Mean field spin glasses treated with {PDE} techniques},
   JOURNAL = {Eur. Phys. J. B},
    VOLUME = {86},
      YEAR = {2013},
    NUMBER = {7},
     PAGES = {Art. 332, 10},
      ISSN = {1434-6028},
   MRCLASS = {82D30},
  MRNUMBER = {3089483},
       DOI = {10.1140/epjb/e2013-40334-6},
       URL = {https://doi.org/10.1140/epjb/e2013-40334-6},
}

@article {bovklim,
    AUTHOR = {Bovier, Anton and Klimovsky, Anton},
     TITLE = {The {Aizenman--Sims--Starr} and {G}uerra's schemes for the
              {SK} model with multidimensional spins},
   JOURNAL = {Electron. J. Probab.},
    VOLUME = {14},
      YEAR = {2009},
     PAGES = {no. 8, 161--241},
   MRCLASS = {82B44 (60F10 60K35)},
  MRNUMBER = {2471664},
MRREVIEWER = {Nicoletta Cancrini},
       DOI = {10.1214/EJP.v14-611},
       URL = {https://doi.org/10.1214/EJP.v14-611},
}

@article {gue03,
    AUTHOR = {Guerra, Francesco},
     TITLE = {Broken replica symmetry bounds in the mean field spin glass
              model},
   JOURNAL = {Comm. Math. Phys.},
    VOLUME = {233},
      YEAR = {2003},
    NUMBER = {1},
     PAGES = {1--12},
      ISSN = {0010-3616},
   MRCLASS = {82B44 (82D30)},
  MRNUMBER = {1957729},
MRREVIEWER = {Beat M. Niederhauser},
       DOI = {10.1007/s00220-002-0773-5},
       URL = {https://doi.org/10.1007/s00220-002-0773-5},
}

@article {jagtob,
    AUTHOR = {Jagannath, Aukosh and Tobasco, Ian},
     TITLE = {A dynamic programming approach to the {Parisi} functional},
   JOURNAL = {Proc. Amer. Math. Soc.},
    VOLUME = {144},
      YEAR = {2016},
    NUMBER = {7},
     PAGES = {3135--3150},
      ISSN = {0002-9939},
   MRCLASS = {60K35 (35K58 35Q82 49L20 82B44 82D30)},
  MRNUMBER = {3487243},
MRREVIEWER = {Giuseppe Genovese},
       DOI = {10.1090/proc/12968},
       URL = {https://doi.org/10.1090/proc/12968},
}

@article {mourrat2020extending,
    AUTHOR = {Mourrat, Jean-Christophe and Panchenko, Dmitry},
     TITLE = {Extending the {Parisi} formula along a {Hamilton--Jacobi}
              equation},
   JOURNAL = {Electron. J. Probab.},
    VOLUME = {25},
      YEAR = {2020},
     PAGES = {Paper No. 23, 17},
   MRCLASS = {82D30 (60G15 82B44)},
  MRNUMBER = {4073684},
       DOI = {10.1214/20-ejp432},
       URL = {https://doi.org/10.1214/20-ejp432},
}

@article{mourrat2021nonconvex,
    AUTHOR = {Mourrat, Jean-Christophe},
     TITLE = {Nonconvex interactions in mean-field spin glasses},
   JOURNAL = {Probab. Math. Phys.},
    VOLUME = {2},
      YEAR = {2021},
    NUMBER = {2},
     PAGES = {281--339},
      ISSN = {2690-0998,2690-1005},
   MRCLASS = {82D30 (82B44)},
  MRNUMBER = {4408014},
       DOI = {10.2140/pmp.2021.2.281},
       URL = {https://doi.org/10.2140/pmp.2021.2.281},
}

@article {pan.aom,
    AUTHOR = {Panchenko, Dmitry},
     TITLE = {The {Parisi} ultrametricity conjecture},
   JOURNAL = {Ann. of Math. (2)},
    VOLUME = {177},
      YEAR = {2013},
    NUMBER = {1},
     PAGES = {383--393},
      ISSN = {0003-486X},
   MRCLASS = {60G57 (82D30)},
  MRNUMBER = {2999044},
MRREVIEWER = {Flora Koukiou},
       DOI = {10.4007/annals.2013.177.1.8},
       URL = {https://doi.org/10.4007/annals.2013.177.1.8},
}

@book {pan,
    AUTHOR = {Panchenko, Dmitry},
     TITLE = {The {Sherrington--Kirkpatrick} model},
    SERIES = {Springer Monographs in Mathematics},
 PUBLISHER = {Springer, New York},
      YEAR = {2013},
     PAGES = {xii+156},
      ISBN = {978-1-4614-6288-0; 978-1-4614-6289-7},
   MRCLASS = {82-02 (82B44 82D30)},
  MRNUMBER = {3052333},
MRREVIEWER = {Jos\'{e} Trashorras},
       DOI = {10.1007/978-1-4614-6289-7},
       URL = {https://doi.org/10.1007/978-1-4614-6289-7},
}

@article {pan.multi,
    AUTHOR = {Panchenko, Dmitry},
     TITLE = {The free energy in a multi-species {Sherrington--Kirkpatrick}
              model},
   JOURNAL = {Ann. Probab.},
    VOLUME = {43},
      YEAR = {2015},
    NUMBER = {6},
     PAGES = {3494--3513},
      ISSN = {0091-1798},
   MRCLASS = {60K35 (60G09 82B44)},
  MRNUMBER = {3433586},
MRREVIEWER = {Rongfeng Sun},
       DOI = {10.1214/14-AOP967},
       URL = {https://doi.org/10.1214/14-AOP967},
}

@article {pan.vec,
    AUTHOR = {Panchenko, Dmitry},
     TITLE = {Free energy in the mixed {$p$}-spin models with vector spins},
   JOURNAL = {Ann. Probab.},
    VOLUME = {46},
      YEAR = {\noop{2019}2018},
    NUMBER = {2},
     PAGES = {865--896},
      ISSN = {0091-1798},
   MRCLASS = {60F10 (60G15 60K35 82B44 82D30)},
  MRNUMBER = {3773376},
       DOI = {10.1214/17-AOP1194},
       URL = {https://doi.org/10.1214/17-AOP1194},
}

@article{parisi79,
  title={Infinite number of order parameters for spin-glasses},
  author={Parisi, Giorgio},
  journal={Phys. Rev. Lett.},
  volume={43},
  number={23},
  pages={1754},
  year={1979},
  publisher={APS}
}

@article{parisi80,
  title={A sequence of approximated solutions to the {SK} model for spin glasses},
  author={Parisi, Giorgio},
  journal={J. Phys. A: Math. Gen.},
  volume={13},
  number={4},
  pages={L115--L121},
  year={1980},
  publisher={IOP Publishing}
}

@article {Tpaper,
    AUTHOR = {Talagrand, Michel},
     TITLE = {The {Parisi} formula},
   JOURNAL = {Ann. of Math. (2)},
    VOLUME = {163},
      YEAR = {2006},
    NUMBER = {1},
     PAGES = {221--263},
      ISSN = {0003-486X},
   MRCLASS = {82B44 (82D30)},
  MRNUMBER = {2195134},
MRREVIEWER = {Fran\c{c}ois Germinet},
       DOI = {10.4007/annals.2006.163.221},
       URL = {https://doi.org/10.4007/annals.2006.163.221},
}

@article {tal.sph,
    AUTHOR = {Talagrand, Michel},
     TITLE = {Free energy of the spherical mean field model},
   JOURNAL = {Probab. Theory Related Fields},
    VOLUME = {134},
      YEAR = {2006},
    NUMBER = {3},
     PAGES = {339--382},
      ISSN = {0178-8051},
   MRCLASS = {82B44 (60K40)},
  MRNUMBER = {2226885},
MRREVIEWER = {Anton Bovier},
       DOI = {10.1007/s00440-005-0433-8},
       URL = {https://doi-org.proxy.rubens.ens.fr/10.1007/s00440-005-0433-8},
}

@article{mourrat2020free,
  title={Free energy upper bound for mean-field vector spin glasses},
  author={Mourrat, Jean-Christophe},
  journal={Preprint, arXiv:2010.09114},
  year={2020}
}

@article{genovese2009mechanical,
  title={A mechanical approach to mean field spin models},
  author={Genovese, Giuseppe and Barra, Adriano},
  journal={J. Math. Phys.},
  volume={50},
  number={5},
  pages={053303},
  year={2009},
  publisher={American Institute of Physics}
}

@article{chen2022hamilton,
  title={{Hamilton--Jacobi equations from mean-field spin glasses}},
  author={Chen, Hong-Bin and Xia, Jiaming},
  journal={Probab. Theory Related Fields},
  year={2025},
  volume={192},
  number={3},
  pages={803--873},
}

@book{HJbook,
    Author = {Tomas Dominguez and Jean-Christophe Mourrat},
    Title = {Statistical mechanics of mean-field disordered systems: a {Hamilton--Jacobi} approach},
    SERIES = {Zurich Lectures in Advanced Mathematics},
    PUBLISHER = {EMS Press},
    YEAR = {2024}
}

@article{auffinger2015properties,
  title={{On properties of {Parisi} measures}},
  author={Auffinger, Antonio and Chen, Wei-Kuo},
  journal={Probab. Theory Related Fields},
  volume={161},
  number={3},
  pages={817--850},
  year={2015},
  publisher={Springer}
}

@book{karatzas1991brownian,
  title={Brownian Motion and Stochastic Calculus},
  author={Karatzas, I. and Shreve, S.},
  isbn={9780387976556},
  lccn={96167783},
  series={Graduate Texts in Mathematics (113) (Book 113)},
  url={https://books.google.fr/books?id=ATNy_Zg3PSsC},
  year={1991},
  publisher={Springer New York}
}

@article{chen2023parisi,
      title={{Parisi} {PDE} and convexity for vector spins}, 
      author={Hong-Bin Chen},
      journal = {Preprint, arXiv:2311.10446},
      year={2023}
}

@article{bates2023parisi,
  author  = {Bates, Erik and Sohn, Youngtak},
  title   = {{Parisi} formula for balanced {Potts} spin glass},
  journal = {Comm. Math. Phys.},
  volume  = {405},
  number  = {10},
  pages   = {Paper No. 228, 68},
  year    = {2024},
  doi     = {10.1007/s00220-024-05100-9},
  url     = {https://doi.org/10.1007/s00220-024-05100-9}
}

@article{chen2024selfoverlap,
  author  = {Hong-Bin Chen},
  title   = {On the self-overlap in vector spin glasses},
  journal = {Journal of Mathematical Physics},
  volume  = {65},
  number  = {3},
  pages   = {031901},
  year    = {2024}
}

@article{chen2024simultaneous,
  title={Simultaneous replica-symmetry breaking for vector spin glasses},
  author={Chen, Hong-Bin and Mourrat, Jean-Christophe},
  journal={Preprint, arXiv:2411.14105},
  year={2024}
}

@article {auffinger2015parisi,
    AUTHOR = {Auffinger, Antonio and Chen, Wei-Kuo},
     TITLE = {The {Parisi} formula has a unique minimizer},
   JOURNAL = {Comm. Math. Phys.},
    VOLUME = {335},
      YEAR = {2015},
    NUMBER = {3},
     PAGES = {1429--1444},
      ISSN = {0010-3616},
   MRCLASS = {82D30},
  MRNUMBER = {3320318},
       DOI = {10.1007/s00220-014-2254-z},
       URL = {https://doi.org/10.1007/s00220-014-2254-z},
}

@article{mourrat2023free,
     TITLE = {Free energy upper bound for mean-field vector spin glasses},
    AUTHOR = {Mourrat, Jean-Christophe},
   JOURNAL = {Ann. Inst. Henri Poincar\'{e} Probab. Stat.},
    VOLUME = {59},
      YEAR = {2023},
    NUMBER = {3},
     PAGES = {1143--1182},
}

@article{chen2025free,
    AUTHOR = {Chen, Hong-Bin and Mourrat, Jean-Christophe},
     TITLE = {On the free energy of vector spin glasses with nonconvex
              interactions},
   JOURNAL = {Probab. Math. Phys.},
    VOLUME = {6},
      YEAR = {2025},
    NUMBER = {1},
     PAGES = {1--80},
      ISSN = {2690-0998,2690-1005},
   MRCLASS = {82D30 (82B44)},
  MRNUMBER = {4848391},
MRREVIEWER = {Xingbin\ Pan},
       DOI = {10.2140/pmp.2025.6.1},
       URL = {https://doi.org/10.2140/pmp.2025.6.1},
}

@article{bates2025balanced,
  title   = {Balanced multi-species spin glasses},
  author  = {Bates, Erik and Sohn, Youngtak},
  journal = {Preprint, arXiv:2507.06522},
  year    = {2025}
}

@article {bates2022free,
    AUTHOR = {Bates, Erik and Sohn, Youngtak},
     TITLE = {Free energy in multi-species mixed {$p$}-spin spherical
              models},
   JOURNAL = {Electron. J. Probab.},
    VOLUME = {27},
      YEAR = {2022},
     PAGES = {Paper No. 52, 75},
      ISSN = {1083-6489},
   MRCLASS = {82B44 (60G15 82D30)},
  MRNUMBER = {4416676},
MRREVIEWER = {Flora\ Koukiou},
       DOI = {10.1214/22-ejp780},
       URL = {https://doi.org/10.1214/22-ejp780},
}

@article{chen2026spherical,
  title  = {Free energy of non-convex multi-species spherical spin glasses},
  author = {Chen, Hong-Bin and Mourrat, Jean-Christophe},
  journal={Preprint, arXiv:2609.32710},
  year   = {2026}
}

@article{mourrat2022parisi,
    AUTHOR = {Mourrat, Jean-Christophe},
     TITLE = {The {Parisi} formula is a {Hamilton--Jacobi} equation in
              {Wasserstein} space},
   JOURNAL = {Canad. J. Math.},
    VOLUME = {74},
      YEAR = {2022},
    NUMBER = {3},
     PAGES = {607--629},
      ISSN = {0008-414X,1496-4279},
   MRCLASS = {82B44 (82D30)},
  MRNUMBER = {4430924},
       DOI = {10.4153/s0008414x21000031},
       URL = {https://doi.org/10.4153/s0008414x21000031},
}

@article{chen2026ising,
  author        = {Chen, Hong-Bin and Issa, Victor and Mourrat, Jean-Christophe},
  title         = {Free energy of non-convex multi-species spin glasses with centered {Ising} spins},
  journal       = {Preprint, arXiv:2606.16636},
  year          = {2026},
  eprint        = {2606.16636},
  archivePrefix = {arXiv},
  primaryClass  = {math.PR},
  doi           = {10.48550/arXiv.2606.16636},
  url           = {https://arxiv.org/abs/2606.16636}
}

@article{ho2026convergence,
  author        = {Ho, Fu-Hsuan},
  title         = {Convergence of the free energy of vector spin glasses with non-convex interactions},
  journal       = {Preprint, arXiv:2609.39307},
  year          = {2026},
  eprint        = {2609.39307},
  archivePrefix = {arXiv},
  primaryClass  = {math.PR}
}

@article{brascampLieb1976,
  author = {Brascamp, Herm Jan and Lieb, Elliott H.},
  title = {On extensions of the {Brunn--Minkowski} and {Pr\'ekopa--Leindler} theorems, including inequalities for log concave functions, and with an application to the diffusion equation},
  journal = {J. Funct. Anal.},
  volume = {22},
  number = {4},
  pages = {366--389},
  year = {1976},
  doi = {10.1016/0022-1236(76)90004-5}
}

@article{eldan2020chaos,
  author  = {Eldan, Ronen},
  title   = {A simple approach to chaos for {$p$}-spin models},
  journal = {J. Stat. Phys.},
  volume  = {181},
  pages   = {1266--1276},
  year    = {2020},
  doi     = {10.1007/s10955-020-02626-9},
  eprint  = {2004.14885},
  archivePrefix = {arXiv}
}

@article{frohlichZegarlinski1989,
  author  = {Fr{\"o}hlich, J{\"u}rg and Zegarlinski, Bogus{\l}aw},
  title   = {Spin glasses and other lattice systems with long range interactions},
  journal = {Comm. Math. Phys.},
  volume  = {120},
  number  = {4},
  pages   = {665--688},
  year    = {1989},
  doi     = {10.1007/BF01260392}
}

@article{barbierPanchenkoSaenz2022,
  author  = {Barbier, Jean and Panchenko, Dmitry and S{\'a}enz, Manuel},
  title   = {Strong replica symmetry for high-dimensional disordered log-concave {Gibbs} measures},
  journal = {Information and Inference: A Journal of the IMA},
  volume  = {11},
  number  = {3},
  pages   = {1079--1108},
  year    = {2022},
  doi     = {10.1093/imaiai/iaab027},
  eprint  = {2009.12939},
  archivePrefix = {arXiv}
}

@article{chen2017variational,
  author  = {Chen, Wei-Kuo},
  title   = {Variational representations for the {Parisi} functional and the two-dimensional {Guerra--Talagrand} bound},
  journal = {Ann. Probab.},
  volume  = {45},
  number  = {6A},
  pages   = {3929--3966},
  year    = {2017},
  doi     = {10.1214/16-AOP1154},
  eprint  = {1501.06635},
  archivePrefix = {arXiv}
}

@book{durrett2019probability,
  title={Probability: Theory and Examples},
  author={Durrett, Rick},
  edition={5},
  year={2019},
  publisher={Cambridge University Press},
  doi={10.1017/9781108591034}
}

@misc{AronowLopatto2026,
  author       = {Aronow, P. M. and Lopatto, Patrick},
  title        = {Quantitative {Parisi} formulas and fluctuations in the {Sherrington--Kirkpatrick} model},
  year         = {2026},
  eprint       = {2609.09103v1},
  eprinttype   = {arxiv},
  eprintclass  = {math.PR}
}
\end{document}